\documentclass[11pt,twoside]{article}%
\usepackage[colorlinks=true, bookmarksnumbered=true, bookmarksopen=true,
bookmarksopenlevel=3, pdfstartview=FitH, linkcolor=cyan, pdfmenubar=true,
pdftoolbar=true, bookmarks=true,citecolor=cyan, urlcolor=magenta,
filecolor=cyan,menucolor=black,plainpages=false,pdfpagelabels]{hyperref}
\usepackage[lmargin=2cm,tmargin=2cm,bmargin=2cm,rmargin=2cm]{geometry}
\usepackage{lineno,hyperref}
\usepackage{amsfonts}
\usepackage[latin1]{inputenc}
\usepackage{amssymb,latexsym}
\usepackage{amsmath,amsthm}
\usepackage[Sonny]{fncychap}
\usepackage{enumerate}
\usepackage[dvips]{color,graphicx}
\usepackage{amsbsy,amssymb,amsmath}
\usepackage{fancyhdr}
\usepackage{graphicx,color}
\usepackage{enumerate}
\usepackage{mathtools}
\usepackage{tikz}
\usepackage{pstricks}
\usepackage{pstricks-add}
\usepackage{pst-all,pst-math,pst-plot}
\usepackage{textcomp}
\usepackage{amsmath}
\usepackage{amssymb}
\usepackage{graphicx}%
\newtheorem{theorem}{Theorem}[section]

\newtheorem{lemma}[theorem]{Lemma}
\newtheorem{remark}{Remark}[section]

\numberwithin{equation}{section}

\begin{document}

\title{On the global well-posedness and self-similar solutions for a nonlinear
elliptic problem with a dynamic boundary condition}
\author{{{Lucas C. F. Ferreira}{\thanks{LCFF was supported by CNPq, grant:
312484/2023-2, Brazil. (corresponding author) }}}\\{\small State University of Campinas, IMECC-Department of Mathematics,} \\{\small {Rua S\'{e}rgio Buarque de Holanda, 651, CEP 13083-859, Campinas-SP,
Brazil.}}\\{\small \texttt{Email:lcff@ime.unicamp.br}}\vspace{0.5cm}\\{{Narayan V. Machaca-Le\'{o}n} {\thanks{NVML was supported by CAPES (Finance Code
001), grant: 88887.682784/2022, Brazil.}}}\\{\small State University of Campinas, IMECC-Department of Mathematics,} \\{\small {Rua S\'{e}rgio Buarque de Holanda, 651, CEP 13083-859, Campinas-SP,
Brazil.}}\\{\small \texttt{Email:nvmleon@gmail.com}}}
\date{}
\maketitle

\begin{abstract}
We are concerned with a semilinear elliptic equation in the half-space,
subject to a nonlinear dynamic boundary condition. We establish the global
well-posedness of solutions in a new setting for the problem, namely the
framework of Morrey spaces. These are strictly larger than $L^{p}$ and
weak-$L^{p}$ spaces, accommodating a broader class of rough initial data,
including homogeneous and nondecaying (at infinity) profiles. In our analysis,
we consider functional spaces invariant under the natural scaling of the
problem, which enables the construction of self-similar solutions. To achieve
this, we need to derive key estimates in Morrey spaces for certain interior
and boundary operators that arise from the corresponding integral formulation.
Furthermore, we obtain some qualitative properties of the solutions, such as
positivity, symmetry, and asymptotic stability. Leveraging this last property,
we show the existence of self-similar attractor basins and construct a class
of solutions that are asymptotically self-similar.

\

\noindent\textbf{AMS MSC:} 35J61; 35J66; 35A01; 35A02; 35B06; 35B07; 35C06;
35B40; 42B35 \vspace{.2cm}

\noindent\textbf{Keywords:} Nonlinear elliptic equations; Dynamic boundary
conditions; Well-posedness; Self-similar solutions; Symmetries; Asymptotic
stability; Morrey spaces

\end{abstract}

\pagestyle{myheadings} \markboth{ \small{{\sc l.c.f. ferreira  and  n.v. machaca-le\'{o}n}\hspace{6.0cm}}}{\small{\hspace{3.6cm}{\sc An elliptic equation with dynamic boundary condition}}}

\section{Introduction}

\label{Sec1}\hspace{0.3cm}We consider the semilinear elliptic problem with a
dynamic boundary condition in the half-space $\mathbb{R}_{+}^{n}$
\begin{equation}%
\begin{cases}
-\Delta u=u|u|^{p_{1}-1}, & x\in\mathbb{R}_{+}^{n},\text{ }t>0,\\
\partial_{t}u+\partial_{\nu}u=u|u|^{p_{2}-1}, & x\in\partial\mathbb{R}_{+}%
^{n},\text{ }t>0.\\
u(x,0)=\varphi(x^{\prime}), & x=(x^{\prime},0)\in\partial\mathbb{R}_{+}^{n},
\end{cases}
\label{P1}%
\end{equation}
where $n\geq3$, $p_{1},p_{2}>1$, $\mathbb{R}_{+}^{n}=\{(x^{\prime},x_{n}%
)\in\mathbb{R}^{n}:x_{n}>0\}$, and $\partial\mathbb{R}_{+}^{n}=\mathbb{R}%
^{n-1}.$ The operator $\Delta$ denotes the Laplacian with respect to the
spatial variables, $\partial_{t}=\partial/\partial t$, and $\partial_{\nu
}=-\partial/\partial x_{n}$ is the outward normal derivative. The initial data
$\varphi(x^{\prime})$ prescribes the boundary condition at time $t=0$.

Problems involving semilinear elliptic equations with dynamic boundary
conditions arise naturally in various physical models, such as heat conduction
with surface effects, diffusion processes with memory, diffusion of solute
from a well-stirred fluid, reaction-diffusion systems in heterogeneous media
(see \cite{Crank1975},\cite{fila2015existence},\cite{giga2018dynamic}%
,\cite{Latorre-Segura2018} and their references). The inclusion of a time
derivative on the boundary reflects the presence of inertia or storage effects
at the interface, leading to a richer and more realistic modeling of dynamic
interactions between the domain and its boundary. In comparison with their
stationary counterpart, they remain less understood and continue to attract
considerable attention, especially in the nonlinear regime.

The study of such problems poses several mathematical challenges due to the
interplay between the elliptic nature of the equation in the domain and the
parabolic (or pseudo-parabolic) behavior induced by the boundary condition. In
particular, the nonlinear structure of the boundary condition introduces
significant difficulties in the analysis of well-posedness, blow-up phenomena,
and long-time behavior of solutions, among other properties.

A related problem is given by
\begin{equation}%
\begin{cases}
-\Delta u=f(x,u), & x\in\Omega,t>0\\
\partial_{t}u+\partial_{\nu}u=g(x,u), & x\in\Omega,t>0\\
u(x,0)=\varphi(x^{\prime}), & x=(x^{\prime},0)\in\partial\Omega
\end{cases}
\label{P2}%
\end{equation}
where $f$ and $g$ are given functions and $\Omega\subset\mathbb{R}^{n}$ is a
smooth domain. This type of system has been studied from both an abstract and
applied perspective. A classical framework for analyzing nonlinear boundary
value problems of this type is presented in the seminal work of
Lions~\cite{lions1969quelques}, where the author studied (\ref{P2})~in bounded
domains $\Omega$, specifically for the case $f=0$, $g(u)=-u\left\vert
u\right\vert ^{p-1}$ and $\varphi\in H_{2,2}^{1/2}(\partial\Omega)$, obtaining
existence-uniqueness results by means of the theory of maximal monotone
operators. Further development of this approach for a problem involving a
dynamic boundary condition with a second-order time derivative $\partial
_{tt}u$ was carried out by D\'{\i}az and Jim\'{e}nez \cite{diaz1984aplicacion}%
. A more conceptual contribution is due to T. Hintermann
\cite{hintermann1989evolution}, who formalized the notion of dynamic boundary
operators within the framework of semigroups in the context of Sobolev
$W_{p}^{s}$ and Besov $B_{p,p}^{s}$ spaces. Considering $f=f(x)$ and $g=g(x)$
(linear case) and $\varphi\in$ $B_{p,p}^{2-1/p}(\partial\Omega)$ with
$1<p<\infty,$ he proved that (\ref{P2}) is well-posed in the Sobolev space
$W_{p}^{2}(\Omega)$ by employing pseudodifferential operators and
Mikhlin-H\"{o}rmander multiplier theorem. The study of nonlinear elliptic
systems with dynamic boundary conditions was advanced by Escher
\cite{escher1992nonlinear}, where the author established local and global
existence-uniqueness results of $W_{p}^{1}(\Omega)$-weak solutions by
considering bounded domains, suitable growth conditions on the nonlinearities,
and initial data $\varphi\in B_{p,p}^{1-1/p}(\partial\Omega)$ for $p>n$.
Further studies by Kirane \cite{kirane1992blow} and Kirane, Nabana, and
Pohozaev \cite{kirane2004nonexistence} focused on blow-up phenomena and
nonexistence results for problems with semilinear dynamic boundary conditions
of parabolic and hyperbolic type in the framework of Bessel potential and
Besov spaces. A Fujita-type result in the setting of nonnegative bounded and
uniformly continuous solutions, i.e. $BUC_{+}(\overline{\Omega})$, was
established by Amann and Fila \cite{Fila1997} by considering (\ref{P2}) with
$\Omega=\mathbb{R}_{+}^{n}$, $f=0$, $g=u^{p},$ and $\varphi\in BUC_{+}%
(\overline{\Omega})$. Considering also the case $f=0,$ Fila and Quittner
\cite{fila1997global} provided valuable insights into global solutions of the
Laplace equation with nonlinear boundary conditions and their connection with
the stationary regime. They investigated the boundedness and derived a priori
estimates for global solutions in bounded domains $\Omega,$ assuming
$\varphi\in H_{2}^{1/2}(\partial\Omega)$ and $g=g(u)$ exhibits power-type
growth in $u$ of order $p$, with $1<p<(3n+1)/(3n-5)$ and $n\geq2.$ These
results were then used by them to establish the existence of sign-changing
stationary solutions. More recently, a series of contributions by Fila, Ishige
and Kawakami have significantly expanded the theory. Their works addressed
various aspects such as existence, uniqueness, minimal and positive solutions,
and asymptotic behavior of the problem (see \cite{fila2013large}%
,\cite{fila2015existence},\cite{fila2016minimal},\cite{fila2017exterior}). In
particular, considering problem (\ref{P2}) with $f=u^{p}$ and $g=0,$ the works
\cite{fila2013large},\cite{fila2015existence},\cite{fila2016minimal} offer
comprehensive treatments of the half-space case $\Omega=\mathbb{R}_{+}^{n}$,
including existence and non-existence results, time and spatial decay
properties, and asymptotic behavior and structural characterizations of
positive solutions. In a related direction, Giga and Hamamuki
\cite{giga2018dynamic} investigated dynamic boundary conditions for singular
degenerate parabolic equations (see also \cite{Giga-DIE-2021}). Their work
develops a theoretical framework in the half-space setting and establishes
results on existence, uniqueness, and asymptotic behavior (w.r.t a parameter
in the boundary condition) of viscosity solutions for uniformly continuous
initial data, i.e., $\varphi\in UC(\overline{\mathbb{R}_{+}^{n}})$. The
contribution by Latorre and Segura de Le\'{o}n \cite{Latorre-Segura2018}
addressed an elliptic problem involving the 1-Laplacian operator with a
dynamic boundary condition in a bounded domain $\Omega$. Their analysis, based
on nonlinear semigroup theory, considered $L^{2}(\partial\Omega)$ boundary
data and established existence and uniqueness results in the functional class
$L_{loc}^{2}((0,\infty);L^{2}(\Omega))\cap L_{loc}^{\infty}((0,\infty
);BV(\Omega))$ with the boundary regularity $C([0,\infty);L^{2}(\partial
\Omega))\cap W_{loc}^{1,1}((0,\infty);L^{2}(\partial\Omega)).$ Considering
also a bounded domain $\Omega$ and suitable growth assumptions on
nonlinearities $f$ and $g$ (including dependence on time and $\nabla u$),
Arrieta, Quittner, and Rodr\'{\i}guez-Bernal \cite{Arrieta-DIE2001} obtained
well-posedness results for a parabolic version of (\ref{P2}) with the
perturbed operator $(-\Delta u+\omega u,\partial_{\nu}u+\omega u)$ for
$\omega>0$ and initial data in Lebesgue spaces $L_{q_{1}}(\Omega)\times
L_{q_{2}}(\partial\Omega)$, and even for finite measure data $\mathcal{M}%
(\Omega)\times\mathcal{M}(\partial\Omega).$ They also studied the elliptic
counterpart of this perturbed problem with $f=0$, obtaining well-posedness
results for boundary data in the Besov space $B_{q}^{1/q}(\partial\Omega)$,
$q\in(1,\infty)$, as well as in $\mathcal{M}(\partial\Omega)$. In
\cite{Binz-2024}, Binz and ter Elst investigated a class of second-order
elliptic operators $Au$ in divergence form in a bounded domain $\Omega$ with
$C^{1,\alpha}$-boundary, where the coefficients are only H\"{o}lder
continuous, and considered the problem $u_{t}+Au=0$ in $\Omega$ with the
dynamic boundary condition $(Tr(u))_{t}+Tr(Au)=0$ on $\partial\Omega$, and
initial data $u(x,0)=u_{0}$ in $\Omega.$ The boundary conditions are of
Wentzell type and take the form $Tr(Au)=\varphi_{1}\partial_{\nu}u+\varphi
_{2}Tr(u)$ on $\partial\Omega,$ where $\varphi_{1}$ and $\varphi_{2}$ are
bounded measurable functions satisfying suitable hypotheses. Within this
framework, they proved maximal regularity results for $A$ in $L^{q}%
(\Omega)\times L^{q}(\partial\Omega)$ with $q\in(1,\infty)$, and also showed
that $A$ generates a holomorphic $C_{0}$-semigroup of angle $\pi/2$ for
$q\in\lbrack1,\infty).$

In the present work, we investigate problem (\ref{P1}) from a new perspective,
distinct from previous approaches, by exploring it in a different framework
for problems with dynamics boundary conditions. Our focus is on analyzing the
global-in-time well-posedness (in the Hadamard sense) and qualitative
properties of solutions in a setting that accommodates self-similar solutions
and novel classes of rough initial data. To this end, we employ the framework
of Morrey spaces which were originally introduced to study the local behavior
of solutions to elliptic PDEs and extend classical Lebesgue spaces by
incorporating both local and global integrability conditions. In fact, they
are strictly larger than $L^{p}$ and weak-$L^{p}$ spaces, and then allowing
for a broader class of functions to be considered. This makes Morrey spaces
especially well-suited for studying problems in the presence of rough initial
or boundary data, as well as nondecaying data (i.e., data that do not decay at
infinity). Moreover, the structure of Morrey spaces is compatible with the
scaling properties of many PDEs, enabling the construction and analysis of
self-similar solutions.

Considering $p_{1}=2p_{2}-1$, we have the scaling map for (\ref{P1})
\begin{equation}
u(x,t)\mapsto u_{\lambda}(x,t)=\lambda^{\frac{2}{p_{1}-1}}u(\lambda x,\lambda
t)=\lambda^{\frac{1}{p_{2}-1}}u(\lambda x,\lambda t), \label{sc1}%
\end{equation}
which naturally induces a corresponding scaling for the initial-boundary data
$\varphi(x^{\prime})$
\begin{equation}
\varphi(x^{\prime})\mapsto\varphi_{\lambda}(x^{\prime})=\lambda^{\frac
{1}{p_{2}-1}}\varphi(\lambda x^{\prime}). \label{sc2}%
\end{equation}
Solutions invariant under (\ref{sc1}) are referred to as self-similar
solutions. These correspond to data $\varphi(x^{\prime})$ that are homogeneous
functions of degree $-1/(p_{2}-1).$

Problem (\ref{P1}) can be reformulated as an equivalent integral equation (see
\cite{Fila1997},\cite[Remark 1.1]{fila2013large} for details)%
\begin{equation}
u=I_{1}[\varphi]+I_{2}[u_{0}|u_{0}|^{p_{2}-1}]+I_{3}[u|u|^{p_{1}-1}%
]+I_{4}[u|u|^{p_{1}-1}], \label{int1}%
\end{equation}
where, for a function $h(x^{\prime},x_{n})$, we denote the boundary trace by
$h_{0}(x^{\prime})=(h)_{0}(x^{\prime})=h(x^{\prime},0)$, and the integral
operators are defined as follows:
\begin{align*}
&  I_{1}[\varphi](x^{\prime},x_{n},t)=\int_{\mathbb{R}^{n-1}}P(x^{\prime
}-y^{\prime},x_{n}+t)\varphi(y^{\prime})dy^{\prime},\\
&  I_{2}[u_{0}|u_{0}|^{p_{2}-1}](x^{\prime},x_{n},t)=\int_{0}^{t}%
\int_{\mathbb{R}^{n-1}}P(x^{\prime}-y^{\prime},x_{n}+t-s)u(y^{\prime
},0,s)|u(y^{\prime},0,s)|^{p_{2}-1}dy^{\prime}ds,\\
&  I_{3}[u|u|^{p_{1}-1}](x^{\prime},x_{n},t)=\int_{0}^{t}\int_{\mathbb{R}%
_{+}^{n}}P(x^{\prime}-y^{\prime},x_{n}+y_{n}+t-s)u(y,s)|u(y,s)|^{p_{1}%
-1}dyds,\\
&  I_{4}[u|u|^{p_{1}-1}](x^{\prime},x_{n},t)=\int_{\mathbb{R}_{+}^{n}%
}G(x,y)u(y,t)|u(y,t)|^{p_{1}-1}dy,
\end{align*}
with $P(x^{\prime},x_{n})$ and $G(x,y)$ standing for the Poisson kernel and
Green function for the Laplacian in $\mathbb{R}_{+}^{n}$, respectively, given
by
\begin{equation}
P(x^{\prime},x_{n})=c_{n}\frac{x_{n}}{(x_{n}^{2}+|x^{\prime}|^{2})^{\frac
{n}{2}}}\text{ and }G(x,y)=\frac{c_{n}}{2(n-2)}\left(  \frac{1}{|x-y|^{n-2}%
}-\frac{1}{|x-\tilde{y}|^{n-2}}\right)  , \label{Kernels-P-G}%
\end{equation}
where $c_{n}=\Gamma\left(  \frac{n}{2}\right)  \pi^{-n/2}$ and $\tilde
{y}=(y^{\prime},-y_{n})$. \ A solution of (\ref{int1}) is referred to as a
\textit{integral solution} of problem (\ref{P1}). In equation (\ref{int1}),
the four components can be interpreted as follows: $I_{1}$ represents the
boundary data term, $I_{2}$ the boundary nonlinearity, $I_{3}$ the interior
nonlinearity (evolutionary) component, and $I_{4}$ the interior nonlinearity
(instantaneous) term.

We develop a global-in-time well-posedness theory for (\ref{P1}) within the
framework of Morrey spaces, which enables the treatment of singular boundary
data and self-similar solutions (see Theorems \ref{theo} and
\ref{cor:consequences}). This setting seems to be novel in the context of
problems with dynamic boundary conditions and allows us to handle singular
boundary data exhibiting infinitely many poles---cases not covered by previous
results in the half-space. A central aspect of our analysis is the derivation
of boundary and interior estimates for the integral operators $I_{1}%
,I_{2},I_{3},I_{4}$ in Morrey spaces (see Sections \ref{Sec4.1} and
\ref{Sec4.2}), followed by a contraction argument (see Section \ref{Sec5.1}).
To analyze self-similarity, the indices of the function spaces are chosen so
that their norms are invariant under the scalings (\ref{sc1}) and (\ref{sc2}).
Also, we investigate qualitative properties of solutions such as symmetries
(e.g. invariance around the axis $\overrightarrow{Ox_{n}}$) and positivity,
under suitable assumptions on $\varphi$ (see Theorem \ref{cor:consequences}).
Notably, the smallness condition imposed on $\varphi(x^{\prime})$ is with
respect to the weaker norm of Morrey-spaces, which allows for initial data
with arbitrarily large $L^{p}$ and $H^{s}$-norms to be considered. Moreover,
we establish an asymptotic stability result, showing that small perturbations
of the initial data become negligible as time $t\rightarrow\infty$ (see
Theorem \ref{teo-asy}). As a consequence, we are able to construct a basin of
attraction around each self-similar solution, as well as identify a class of
solutions that are asymptotically self-similar (see Remark \ref{Rem-self}).

To be more precise, we analyze the integral formulation (\ref{int1}) in the
Banach space $\mathcal{X}$ consisting of all Bochner measurable functions
$u:(0,\infty)\rightarrow\mathcal{E}_{q_{0},q_{1},q_{2}},$ where
\begin{equation}
\mathcal{E}_{q_{0},q_{1},q_{2}}=(\mathcal{M}_{q_{0},\mu}(\mathbb{R}_{+}%
^{n})\cap\mathcal{M}_{q_{1},\mu}(\mathbb{R}_{+}^{n}))\times\mathcal{M}%
_{q_{2},\mu}(\partial\mathbb{R}_{+}^{n})\text{ with }q_{0}=\frac{(n-\mu
)(p_{1}-1)}{2} \label{space1}%
\end{equation}
and $p_{1}=2p_{2}-1$, such that%
\begin{equation}
\left\Vert u\right\Vert _{\mathcal{X}}=\sup_{t>0}\Vert u(\cdot,t)\Vert
_{\mathcal{M}_{q_{0},\mu}(\mathbb{R}_{+}^{n})}+\sup_{t>0}t^{\alpha}\Vert
u(\cdot,t)\Vert_{\mathcal{M}_{q_{1},\mu}(\mathbb{R}_{+}^{n})}+\sup
_{t>0}t^{\beta}\Vert u(\cdot,0,t)\Vert_{\mathcal{M}_{q_{2},\mu}(\partial
\mathbb{R}_{+}^{n})}<\infty. \label{norm1}%
\end{equation}
The supremum over $t>0$ in (\ref{norm1}) is taken in the essential sense. The
initial data $\varphi$ is assumed to belong to $\mathcal{M}_{\widetilde{q}%
_{0},\mu}(\mathbb{R}^{n-1})$ with $\widetilde{q}_{0}=\frac{(n-1-\mu)(p_{1}%
-1)}{2}<q_{2}$. We point out that the norm in the space $\mathcal{X}$ is a
time-weighted one of Kato type, which provides suitable control both in the
domain $\mathbb{R}_{+}^{n}$ and on the boundary $\partial\mathbb{R}_{+}^{n}$.
Note that the space $\mathcal{E}_{q_{0},q_{1},q_{2}}$ carries a $\mathcal{M}%
_{q,\mu}$-type information for the boundary trace $u|_{\partial\mathbb{R}%
_{+}^{n}}$, without requiring any positive regularity assumption on $u$. This
makes it particularly useful for handling singular boundary terms. On the
other hand, $L^{p}(\partial\mathbb{R}_{+}^{n})$ contains only trivial
homogeneous functions, which naturally motivates the consideration of larger
spaces that include such singular profiles. In fact, we have the strict
continuous inclusion
\[
L^{p}(\partial\mathbb{R}_{+}^{n})\subset\mathcal{M}_{q,\mu}(\partial
\mathbb{R}_{+}^{n}),
\]
for $1\leq q<p<n-1$ and $\frac{n-1}{p}=\frac{n-1-\mu}{q}$. Moreover,
\[
\left\Vert \,|x^{\prime}|^{-k}\right\Vert _{L^{p}(\partial\mathbb{R}_{+}^{n}%
)}=\infty\text{, for any \thinspace}k\in\mathbb{R}/\{0\}\text{ and }%
p\in\lbrack1,\infty],
\]
while $\,|x^{\prime}|^{-k}\in\mathcal{M}_{q,\mu}(\partial\mathbb{R}_{+}^{n}%
)$,$\ $as well as its translations in the $x^{i}$-directions, provided that
$k=\frac{n-1-\mu}{q}.$ Thus, by choosing a sequence $\{x^{i}\}_{i=1}%
^{l}\subset\partial\mathbb{R}_{+}^{n}$, the data $\varphi(x^{\prime})$ may
exhibit an infinite number of poles on the boundary.

In light of the aspects discussed above, we also mention the work by de
Almeida and Ferreira \cite{Ferreira} which studied the Navier-Stokes equations
in the half-space with Dirichlet boundary conditions involving time-dependent
boundary data in Morrey spaces, employing a potential-type analysis. Although
the boundary condition considered in that work is not of dynamic-type (as in
problems (\ref{P1}) and (\ref{P2})) and does not include nonlinear terms, the
time dependence of the boundary data in \cite{Ferreira} introduces features
that create connections between the boundary problems and motivate approaches
combining integral formulations and critical spaces. Finally, while various
definitions of Morrey spaces can be found in the literature, we adopt here the
notations for Morrey and block spaces used, for instance, in
\cite{Giga-CPDE-1989},\cite{Kato1992},\cite{Ferreira},\cite{Ferreira2}%
,\cite{FerreiraSantana2024}, which are consistent with that presented in
\cite{adams2015morrey}.

The paper is organized as follows. In Section \ref{Sec2}, we recall basic
definitions and properties of Morrey spaces, the Poisson kernel, and the
extension operator. Section \ref{Sec3} introduces certain critical
time-dependent functional spaces and presents the main results. In Section
\ref{Sec4}, we establish key estimates for the Poisson kernel and related
operators arising from the integral formulation (\ref{int1}). With these
estimates in hands, we prove our main result in Section \ref{Sec5}, which is
divided into three subsections: Subsection \ref{Sec5.1} (well-posedness),
Subsection \ref{Sec5.2} (self-similarity and symmetry), and Subsection
\ref{Sec5.3} (asymptotic stability).

\section{Preliminaries}

\label{Sec2}In this section, we fix the notation and recall some properties of
Morrey and block spaces that will be useful in the present work.

\subsection{Morrey and block spaces}

Let $x=(x^{\prime},x_{n})\in\mathbb{R}^{n-1}\times\mathbb{R}$ and consider
$\Omega=\mathbb{R}^{n-1}$ or \ $\Omega=\mathbb{R}_{+}^{n}:=\mathbb{R}%
^{n-1}\times\mathbb{R}_{+}.$ We write $f(x)$ or $f(x^{\prime},x_{n})$ for
$x=(x^{\prime},x_{n})\in\mathbb{R}_{+}^{n}$ and $f(x^{\prime})$ for
$x^{\prime}\in\mathbb{R}^{n-1}$.

For $1\leq q\leq\infty$ and $0\leq\mu<\dim(\Omega)$, the \textit{Morrey space}
$\mathcal{M}_{q,\mu}=\mathcal{M}_{q,\mu}(\Omega)$ consists of all measurable
functions $f$ such that
\[
\Vert f\Vert_{\mathcal{M}_{q,\mu}(\Omega)}:=\sup_{x_{0}\in\Omega,\,r>0}%
r^{-\mu/q}\Vert f\Vert_{L^{q}(\Omega_{r}(x_{0}))}<\infty,
\]
where $\Omega_{r}(x_{0})=\Omega\cap B(x_{0},r)$ and $B(x_{0},r)$ denotes the
open ball centered at $x_{0}$ with radius $r$ in $\mathbb{R}^{n}$ (or in
$\mathbb{R}^{n-1}$ if $\Omega=\mathbb{R}^{n-1}$), that is, in $\mathbb{R}%
^{\dim(\Omega)}$. The pair $(\mathcal{M}_{q,\mu},\Vert\cdot\Vert
_{\mathcal{M}_{q,\mu}(\Omega)})$ is a Banach space. Also, we point out that
$\mathcal{M}_{\infty,\mu}(\Omega)$ should be meant as the space $L^{\infty
}(\Omega)$.

Morrey spaces are not reflexive, but they admit a useful characterization of
their predual via block spaces. This property is particularly helpful for
deriving estimates and analyzing convergence through duality.

For that, first we recall the definition of a $(q,\lambda)$-block. Let
$q^{\prime}$ be the conjugate exponent of $q$. For $1<q<\infty$ and
$0\leq\lambda<\dim(\Omega)$, a function $A\in L_{loc}^{q}(\Omega)$ is a
$(q,\lambda)$-block if there exist $x\in\Omega$ and $R>0$ such that
\[
\mathrm{supp}(A)\subset\Omega_{R}(x_{0})\text{ and }R^{\lambda/q^{\prime}%
}\Vert A\Vert_{L^{q}(\Omega_{R}(x_{0}))}\leq1.
\]
The block space $\mathcal{H}_{q,\lambda}(\Omega)$ consists of all functions
$g$ that can be written as
\[
g=\sum_{i=1}^{\infty}\lambda_{i}A_{i},
\]
where each $A_{i}$ is a $(q,\lambda)$-block and the coefficient sequence
$(\lambda_{i})_{i\in\mathbb{N}}$ belongs to $\ell^{1}$. The norm in
$\mathcal{H}_{q,\lambda}(\Omega)$ is given by
\begin{equation}
\Vert g\Vert_{\mathcal{H}_{q,\lambda}(\Omega)}:=\inf\left\{  \sum
_{i=1}^{\infty}|\lambda_{i}|:g=\sum_{i=1}^{\infty}\lambda_{i}A_{i}\text{ with
}A_{i}\text{ a }(q,\lambda)\text{-block}\right\}  . \label{def_blocks_space}%
\end{equation}
The space $\mathcal{H}_{q,\lambda}(\Omega)$ endowed with $\Vert\cdot
\Vert_{\mathcal{H}_{q,\lambda}(\Omega)}$ is a Banach space. Moreover, we have
the duality property (see \cite[Theorem 347]{adams2015morrey}%
,\cite{almeida2018approximation},\cite{Ferreira2})
\begin{equation}
\mathcal{H}_{q^{\prime},\lambda}(\Omega)^{\ast}=\mathcal{M}_{q,\lambda}%
(\Omega). \label{dualidad_M_H}%
\end{equation}

H\"{o}lder-type inequality works well in the framework of Morrey spaces (see,
e.g., \cite[p.~132]{Kato1992}). This the subject of the next lemma.

\begin{lemma}
Let $1\leq p_{i}\leq\infty$, $0\leq\mu_{i}<n$ for $i=1,2,3$, and suppose
\[
\frac{1}{p_{3}}=\frac{1}{p_{1}}+\frac{1}{p_{2}},\quad\frac{n-\mu_{3}}{p_{3}%
}=\frac{n-\mu_{1}}{p_{1}}+\frac{n-\mu_{2}}{p_{2}}.
\]
Then
\begin{equation}
\Vert fg\Vert_{\mathcal{M}_{p_{3},\mu_{3}}(\Omega)}\leq\Vert f\Vert
_{\mathcal{M}_{p_{1},\mu_{1}}(\Omega)}\Vert g\Vert_{\mathcal{M}_{p_{2},\mu
_{2}}(\Omega)}. \label{Holder1}%
\end{equation}

\end{lemma}

In the sequel we recall a trace type theorem concerning Riesz operators in
Morrey spaces found in \cite[Theorem 1.1]{Liu_xiao}. However, we limit this
theorem to two specific cases, each corresponding to particular fixed
hypotheses. By focusing on these cases, we can avoid verifying all the general
conditions of \cite[Theorem 1.1]{Liu_xiao}. This result is also presented in
the introductory section of \cite{Liu}.

\begin{lemma}
Let $0<\gamma<n$, $1<p_{1},p_{2}<\infty$ and $0\leq\mu<n.$

\begin{enumerate}
\item In addition, suppose that $\mu<n-1$, $\gamma=\frac{n-\mu}{p_{1}}%
-\frac{n-\mu-1}{p_{2}}$, and $n-\gamma p_{1}<n-1$. Then, there exists $C>0$
such that
\begin{equation}
\left\Vert \left(  \frac{1}{\left\vert x\right\vert ^{n-\gamma}}\ast f\right)
(x^{\prime},0)\right\Vert _{\mathcal{M}_{p_{2},\mu}(\mathbb{R}^{n-1})}\leq
C\left\Vert f\right\Vert _{\mathcal{M}_{p_{1},\mu}(\mathbb{R}^{n})},
\label{xiao_liu_n_1}%
\end{equation}
for all $f\in\mathcal{M}_{p_{1},\mu}(\mathbb{R}^{n}).$

\item If $\gamma=\frac{n-\mu}{p_{1}}-\frac{n-\mu}{p_{2}}$, there exists $C>0$
such that
\begin{equation}
\left\Vert \frac{1}{\left\vert x\right\vert ^{n-\gamma}}\ast f\right\Vert
_{\mathcal{M}_{p_{2},\mu}(\mathbb{R}^{n})}\leq C\left\Vert f\right\Vert
_{\mathcal{M}_{p_{1},\mu}(\mathbb{R}^{n})}, \label{xiao_liu_n}%
\end{equation}
for all $f\in\mathcal{M}_{p_{1},\mu}(\mathbb{R}^{n}).$
\end{enumerate}
\end{lemma}

\subsection{Poisson kernel, Green function, and the extension operator}

As noted in the Introduction (see (\ref{Kernels-P-G})), the Poisson kernel
$P(x^{\prime},x_{n})$ in $\mathbb{R}_{+}^{n}$ is defined as
\begin{equation}
P(x^{\prime},x_{n})=c_{n}\frac{x_{n}}{(x_{n}^{2}+|x^{\prime}|^{2})^{\frac
{n}{2}}},\text{ for }x^{\prime}\in\mathbb{R}^{n-1}\text{ and }x_{n}>0\text{,}
\label{P-expression-1}%
\end{equation}
where $c_{n}=\Gamma\left(  \frac{n}{2}\right)  \pi^{-n/2}$ and $n\geq3.$ It
satisfies the normalization property of the integral w.r.t. $x^{\prime}%
$-variable (see \cite[Proposition 5, p.61-62]{Stein})
\begin{equation}
\int_{\mathbb{R}^{n-1}}P(x^{\prime},x_{n})dx^{\prime}=1,\quad\forall x_{n}>0.
\label{P_definicion}%
\end{equation}
To simplify notation, for $t>0$ we define the symmetric extension
\begin{equation}
\bar{P}(x^{\prime},x_{n}+t):=%
\begin{cases}
P(x^{\prime},t+x_{n}), & \text{if }x_{n}>0,\\
P(x^{\prime},t-x_{n}), & \text{if }x_{n}\leq0,
\end{cases}
\label{P_extendido}%
\end{equation}
which allows us to consider the convolution%
\begin{equation}
\bar{P}_{x_{n}+t}\ast f\text{, for all }x_{n}\in\mathbb{R}.
\label{aux-Poisson-ext-conv-1}%
\end{equation}

Let $n\geq3$, $q\geq1$, and $\theta\in\lbrack0,1)$. Then the Poisson kernel
$P$ satisfies the following estimates for $z\in\mathbb{R}^{n-1}$ and $t>0$:
\begin{align}
\left\Vert P(z,x_{n}+t)\right\Vert _{L^{q}((0,\infty),dx_{n})}  &  \leq
\tilde{c}_{n}\frac{1}{|t|^{\theta(n-1-\frac{1}{q})}}\frac{1}{|z|^{(1-\theta
)(n-1-\frac{1}{q})}},\label{eq:P_kernel_est1}\\
|P(z,x_{n}+y_{n}+t)|  &  \leq c_{n}\frac{1}{t^{\theta(n-1)}}\frac{1}%
{|(z,x_{n}-y_{n})|^{(1-\theta)(n-1)}}, \label{eq:P_kernel_est2}%
\end{align}
for all $x_{n},y_{n}>0$.

For that, we first note the elementary inequality
\begin{equation}
\frac{1}{(a^{2}+b^{2})^{q}}\leq\frac{1}{(a^{2\theta}b^{2(1-\theta)})^{q}%
},\quad\label{eq:interpolation_ineq}%
\end{equation}
for all $a,b\in\mathbb{R},\ a\neq0$, $\theta\in\lbrack0,1),$ which follows
from Young's inequality in the form:
\[
a^{2\theta}b^{2(1-\theta)}\leq\theta a^{2}+(1-\theta)b^{2}\leq a^{2}+b^{2}.
\]
\medskip\noindent For estimate (\ref{eq:P_kernel_est1}), using expression
(\ref{P-expression-1}) and the change of variables $x_{n}=\sqrt{t^{2}+|z|^{2}%
}\tan(r)$, we obtain
\begin{align*}
\left\Vert P(z,x_{n}+t)\right\Vert _{L^{q}((0,\infty),dx_{n})}  &
=c_{n}\left\Vert \frac{x_{n}+t}{((x_{n}+t)^{2}+|z|^{2})^{n/2}}\right\Vert
_{L^{q}((0,\infty),dx_{n})}\\
&  \leq c_{n}\left[  \int_{0}^{\infty}\frac{1}{((x_{n}^{2}+t^{2}%
+|z|^{2})^{(n-1)/2})^{q}}\,dx_{n}\right]  ^{1/q}\\
&  =\tilde{c}_{n}\frac{1}{(t^{2}+|z|^{2})^{\frac{n-1}{2}-\frac{1}{2q}}},
\end{align*}
where $\tilde{c}_{n}=c_{n}\left[  \int_{0}^{\pi/2}\cos(r)^{(n-1)q-2}%
\,dr\right]  ^{1/q}$, and the integral converges for $n\geq3$ and $q\geq1$.
Then, applying \eqref{eq:interpolation_ineq}, we deduce the desired estimate.
\medskip\noindent Next, we note the pointwise estimate
\begin{align*}
|P(z,x_{n}+y_{n}+t)|  &  =c_{n}\left\vert \frac{x_{n}+t}{((x_{n}+y_{n}%
+t)^{2}+|z|^{2})^{n/2}}\right\vert \\
&  \leq c_{n}\frac{1}{(t^{2}+(x_{n}+y_{n})^{2}+|z|^{2})^{(n-1)/2}}\\
&  \leq c_{n}\frac{1}{(t^{2}+|(z,x_{n}-y_{n})|^{2})^{(n-1)/2}},
\end{align*}
and then apply (\ref{eq:interpolation_ineq}) in order to obtain
(\ref{eq:P_kernel_est2}).

\medskip Recalling now the second kernel in (\ref{Kernels-P-G}), the Green
function $G$ of the Laplacian operator in $\mathbb{R}_{+}^{n}$ is given by
\[
G(x,y)=\frac{c_{n}}{2(n-2)}\left(  \frac{1}{|x-y|^{n-2}}-\frac{1}{|x-\tilde
{y}|^{n-2}}\right)  ,\quad\tilde{y}=(y^{\prime},-y_{n}),
\]
for $x,y\in\mathbb{R}_{+}^{n}$. Since $|x-\tilde{y}|\geq|x-y|$ for all
$x_{n},y_{n}>0$, it follows that
\begin{equation}
|G(x,y)|\leq\bar{c}_{n}\frac{1}{|x-y|^{n-2}},\text{ where }\bar{c}_{n}%
=\frac{c_{n}}{2(n-2)}. \label{eq:Green_est}%
\end{equation}
\medskip

In what follows, we present some definitions and properties of the Stein
extension $\tilde{u}$ in $\mathbb{R}^{n}$ corresponding to a function $u$
originally defined in $\mathbb{R}_{+}^{n}$. This extension is particularly
useful when applying tools or theorems that require working in the whole space
$\mathbb{R}^{n}$, such as Fourier analysis or elliptic PDE theory, while still
starting with data defined only on a half-space or domain with boundary.

According to \cite[Lemma~1, p.~182]{Stein}, there exists a continuous function
$\psi:[1,\infty)\rightarrow\mathbb{R}$ satisfying the following properties:
$\psi(s)=\mathcal{O}(s^{-N})$ as $s\rightarrow\infty$, for every
$N\in\mathbb{N}$, and
\[
\int_{1}^{\infty}\psi(s)\,ds=1\text{ and }\int_{1}^{\infty}s^{k}%
\psi(s)\,ds=0,\quad\text{for all }k\in\mathbb{N}.
\]

For $u\in\mathcal{M}_{q,\mu}(\mathbb{R}_{+}^{n})$, the extension $\tilde{u}$
of $u$, defined for almost every $\ x\in\mathbb{R}^{n}$, is given by the
linear operator
\begin{equation}
\tilde{u}(x^{\prime},x_{n})=%
\begin{cases}
u(x^{\prime},x_{n}), & \text{if }x_{n}>0,\\
\int_{1}^{\infty}u(x^{\prime},(1-2s)x_{n})\psi(s)\,ds, & \text{if }x_{n}\leq0,
\end{cases}
\label{extension1}%
\end{equation}
provided the integrals in (\ref{extension1}) converge.

Before closing this section, we recall that the extension (\ref{extension1})
preserves $\mathcal{M}_{q,\mu}$-regularity (see \cite[Lemma~3.1]{Ferreira}).
Specifically, let $1\leq q<\infty$ and $0\leq\mu<n.$ Then, we have the
estimate
\begin{equation}
\Vert\tilde{u}\Vert_{\mathcal{M}_{q,\mu}(\mathbb{R}^{n})}\leq C\Vert
u\Vert_{\mathcal{M}_{q,\mu}(\mathbb{R}_{+}^{n})}, \label{cont-ext1}%
\end{equation}
for all $u\in\mathcal{M}_{q,\mu}(\mathbb{R}_{+}^{n}).$

\section{Functional setting and main results}

\label{Sec3}Before presenting our results, we introduce suitable
time-dependent functional spaces in which we seek solutions to equation
(\ref{int1}), with a focus on critical spaces, i.e., those invariant under the
scaling transformation (\ref{sc1}). This invariance property will enable us to
establish the existence of self-similar solutions, that is, solutions which
themselves remain invariant under (\ref{sc1}).

Let $p_{1},p_{2}\in(1,\infty)$, $\mu\in\lbrack0,n-2)$ and consider
\begin{equation}
q_{0}=\frac{(n-\mu)(p_{1}-1)}{2}>1,\text{ }q_{1},q_{2}\in(1,\infty)\text{,
}\alpha=\frac{2}{p_{1}-1}-\frac{n-\mu}{q_{1}}\text{, and }\beta=\frac{1}%
{p_{2}-1}-\frac{n-1-\mu}{q_{2}},\text{ } \label{Cond-H0}%
\end{equation}
where the parameters satisfy the conditions
\begin{equation}
\text{ }\frac{n-\mu}{n-\mu-2}<p_{1}<q_{1},\text{ }p_{2}<q_{2},\text{ and }%
\begin{cases}
0<\alpha<\beta p_{2}<1,\\
0<\beta<\alpha p_{1}<1.
\end{cases}
\label{Cond-H1}%
\end{equation}

As outlined in the Introduction (Section \ref{Sec1}), taking the above
considerations into account, we work within the Banach $\mathcal{X}$
consisting of all Bochner-measurable functions $u:(0,\infty)\rightarrow
\mathcal{E}_{q_{0},q_{1},q_{2}}$, where the space $\mathcal{E}_{q_{0}%
,q_{1},q_{2}}$ is defined in (\ref{space1}) and the norm $\left\Vert
\cdot\right\Vert _{\mathcal{X}}$ is given in (\ref{norm1}). The boundary data
$\varphi$ is assumed to belong to the space $\mathcal{M}_{\widetilde{q}%
_{0},\mu}(\mathbb{R}^{n-1})$, with $\widetilde{q}_{0}=(n-1-\mu)(p_{1}-1)/2$.
In view of (\ref{Cond-H0}) and this choice of $\widetilde{q}_{0}$, the norms
$\Vert\cdot\Vert_{\mathcal{X}}$ and $\left\Vert \cdot\right\Vert
_{\mathcal{M}_{\widetilde{q}_{0},\mu}}$ are invariant under (\ref{sc1}) and
(\ref{sc2}), respectively.

We are now in a position to state our main results.

\begin{theorem}
[Well-posedness]\label{theo} Let $n\geq3$, $p_{1},p_{2}\in(1,\infty),$ $\mu
\in\lbrack0,n-2)$, $\widetilde{q}_{0}=(n-1-\mu)(p_{1}-1)/2>1,$ and let
$q_{0},q_{1},q_{2},\alpha,\beta$ as in (\ref{Cond-H0}). Assume conditions
(\ref{Cond-H1}) and suppose that $p_{1}=2p_{2}-1.$ Then, there exists
$\varepsilon>0$ and $\delta=\delta(\varepsilon)$ ($\delta=C\varepsilon$) such
that problem (\ref{P1}) has a unique integral solution $u\in\mathcal{X}$
satisfying $\Vert u\Vert_{\mathcal{X}}\leq2\varepsilon,$ provided that
$\left\Vert \varphi\right\Vert _{\mathcal{M}_{\widetilde{q}_{0},\mu
}(\mathbb{R}^{n-1})}\leq\delta.$ The data-solution map $[\varphi]\rightarrow
u$ is locally Lipschitz continuous. Moreover $u(x^{\prime},0,t)\rightharpoonup
\varphi$ in the weak-$\ast$ topology of $\mathcal{M}_{\widetilde{q}_{0},\mu
}(\mathbb{R}^{n-1})$ as $t\rightarrow0^{+}$.
\end{theorem}

In the next theorem, we provide results on qualitative properties for the
solutions obtained in Theorem \ref{theo}, such as self-similarity, symmetry,
and positivity.

\begin{theorem}
\label{cor:consequences}Let $u$ be the solution $u$ obtained in
Theorem~\ref{theo}. We have the following properties:

\begin{itemize}
\item[(i)] (Self-similarity) Assume that the boundary data satisfies
$\varphi(x^{\prime})=\lambda^{\frac{1}{p_{2}-1}}\varphi(\lambda x^{\prime})$,
for all $\lambda>0$, $x^{\prime}\in\mathbb{R}^{n-1}$. Then, the corresponding
solution $u$ is self-similar, that is,%
\[
u(x,t)=\lambda^{\frac{1}{p_{2}-1}}u(\lambda x,\lambda t).
\]

\item[(ii)] (Axial symmetry) Let $\varphi$ be radially symmetric, that is
$\varphi(M(x^{\prime}))=\varphi(x^{\prime})$ for all rotation matrices $M$ in
$\mathbb{R}^{n-1}$. Then, the solution $u$ is axially symmetric with respect
to the $x_{n}$-axis, that is,
\[
u(M(x^{\prime}),x_{n},t)=u(x,t).
\]

\item[(iii)] (Positivity) Let $U\subset\mathbb{R}^{n-1}$ be a positive-measure
set. If $\varphi\geq0$ (resp. $\leq0$) a.e. in $\mathbb{R}^{n-1}$ with
$\varphi>0$ (resp. $<0$) a.e. in $U$, then $u$ is positive (resp. negative)
a.e. in $\mathbb{R}_{+}^{n}\times(0,\infty)$ and a.e. in $\mathbb{R}%
^{n-1}\times(0,\infty).$
\end{itemize}
\end{theorem}

\bigskip

We now present an asymptotic stability result, showing that small
perturbations in the boundary data become negligible as time goes to infinity,
leading the solution to converge to the unperturbed configuration.

\begin{theorem}
[Asymptotic stability]\label{teo-asy} Suppose the same hypotheses of Theorem
\ref{theo}. Let $u$ and $v\in\mathcal{X}$ be solutions with the respective
boundary data $\varphi$ and $\psi$ obtained in Theorem \ref{theo}. If the
initial perturbation $\varphi-\psi$ satisfies
\begin{equation}
\lim_{t\rightarrow\infty}\left(  \left\Vert I_{1}[\varphi-\psi](\cdot
,t)\right\Vert _{\mathcal{M}_{q_{0},\mu}(\mathbb{R}_{+}^{n})}+t^{\alpha
}\left\Vert I_{1}[\varphi-\psi](\cdot,t)\right\Vert _{\mathcal{M}_{q_{1},\mu
}(\mathbb{R}_{+}^{n})}+t^{\beta}\left\Vert I_{1}[\varphi-\psi](\cdot
,t)\right\Vert _{\mathcal{M}_{q_{2},\mu}(\mathbb{R}^{n-1})}\right)  =0,
\label{cond-asymp1}%
\end{equation}
then we have that
\begin{equation}
\lim_{t\rightarrow\infty}\left(  \left\Vert (u-v)(\cdot,t)\right\Vert
_{\mathcal{M}_{q_{0},\mu}(\mathbb{R}_{+}^{n})}+t^{\alpha}\left\Vert
(u-v)(\cdot,t)\right\Vert _{\mathcal{M}_{q_{1},\mu}(\mathbb{R}_{+}^{n}%
)}+t^{\beta}\left\Vert (u-v)(\cdot,t)\right\Vert _{\mathcal{M}_{q_{2},\mu
}(\mathbb{R}^{n-1})}\right)  =0. \label{asymp1}%
\end{equation}

\end{theorem}

In light of Theorem \ref{cor:consequences} (i) and Theorem \ref{teo-asy}, we
can construct a class of asymptotically self-similar solutions, as discussed
in the remark below.

\begin{remark}
[Asymptotic self-similarity]\label{Rem-self}Under the hypotheses of Theorems
\ref{theo} and \ref{teo-asy}. Let $u$ be a self-similar solution of the
integral equation (\ref{int1}) with initial boundary data $\varphi
\in\mathcal{M}_{\widetilde{q}_{0},\mu}(\mathbb{R}^{n-1})$. Consider the
perturbation $\psi(x^{\prime})=\varphi(x^{\prime})+a(x^{\prime})$ with
$a\in\mathcal{M}_{\frac{(n-1-\mu)(p_{1}-1)}{2+\gamma},\mu}(\mathbb{R}^{n-1})$
where $\gamma>0$ is such that $\frac{(n-1-\mu)(p_{1}-1)}{2+\gamma}>1$ (recall
that $\widetilde{q}_{0}=(n-1-\mu)(p_{1}-1)/2>1$). It follows that
$a=\psi-\varphi$ verifies (\ref{cond-asymp1}) and, consequently, the solution
$v$ with initial boundary data $\psi$ is attracted to the self-similar
solution $u$ in the sense of (\ref{asymp1}). In particular, we can consider
small perturbations $a(x^{\prime})$ belonging to the Schwartz space
$\mathcal{S}(\mathbb{R}^{n-1})\subset\mathcal{M}_{\frac{(n-1-\mu)(p_{1}%
-1)}{2+\gamma},\mu}(\mathbb{R}^{n-1}).$ Thus, we obtain a basin of attraction
around each self-similar solution and a class of asymptotically self-similar solutions.
\end{remark}

\section{\bigskip Key estimates}

\label{Sec4}This section is devoted to deriving key estimates for the boundary
and interior operators in Morrey spaces that appear in the integral
formulation (\ref{int1}). Moreover, we provide an approximate identity-type
property for the Poisson kernel in the framework of block spaces.

\subsection{Linear estimates}

\label{Sec4.1}We start by providing the estimates for the operator $I_{1}$ in
(\ref{int1}).

\begin{lemma}
\label{Lem-I1}Let $p_{1},p_{2}\in(1,\infty)$, $\mu\in\lbrack0,n-1)$,
$\widetilde{q}_{0}=(n-1-\mu)(p_{1}-1)/2$, and let $q_{0},q_{1},q_{2}%
,\alpha,\beta$ be as in (\ref{Cond-H0}). Assume also that $\widetilde{q}%
_{0}>1$, $p_{1}=2p_{2}-1,$ $\alpha\in\left(  0,n-1-\frac{1}{q_{1}}\right)  $,
and $\beta\in\left(  0,n-1\right)  .$ Then, there exists a constant $C_{1}>0$
such that
\begin{equation}
\left\Vert I_{1}[\varphi]\right\Vert _{\mathcal{X}}\leq C_{1}\left\Vert
\varphi\right\Vert _{\mathcal{M}_{\widetilde{q}_{0},\mu}(\mathbb{R}^{n-1})},
\label{lineal_1}%
\end{equation}
for all $\varphi\in\mathcal{M}_{\widetilde{q}_{0},\mu}(\mathbb{R}^{n-1}).$
\end{lemma}

\textbf{Proof.} First, recall that the $\mathcal{X}$-norm consists of three
components, which will be treated separately. Let $[\Omega_{r}(x_{0}%
)]^{\prime}$ denote the projection of $\Omega_{r}(x_{0})=\mathbb{R}_{+}%
^{n}\cap B(x_{0},r)$ onto $\partial\mathbb{R}_{+}^{n}$. Using properties of
the Poisson kernel $P$, and noting that $\Omega_{r}(x_{0})\subset\lbrack
\Omega_{r}(x_{0})]^{\prime}\times(0,\infty)$, we can estimate%
\begin{align}
t^{\alpha}\left\Vert I_{1}[\varphi](\cdot,t)\right\Vert _{L^{q_{1}}(\Omega
_{r}(x_{0}))}  &  \leq t^{\alpha}\left\Vert \int_{\mathbb{R}^{n-1}}%
P(x^{\prime}-y^{\prime},x_{n}+t)|\varphi(y^{\prime})|dy^{\prime}\right\Vert
_{L^{q_{1}}([\Omega_{r}(x_{0})]^{\prime}\times(0,\infty))}\nonumber\\
&  \leq t^{\alpha}\left\Vert \int_{\mathbb{R}^{n-1}}\left\Vert P(x^{\prime
}-y^{\prime},x_{n}+t)\right\Vert _{L^{q_{1}}((0,\infty),dx_{n})}%
|\varphi(y^{\prime})|dy^{\prime}\right\Vert _{L^{q_{1}}([\Omega_{r}%
(x_{0})]^{\prime},dx^{\prime})}\nonumber\\
&  \leq\tilde{c}_{n}\frac{t^{\alpha}}{t^{\theta\left(  n-1-\frac{1}{q_{1}%
}\right)  }}\left\Vert \int_{\mathbb{R}^{n-1}}\frac{|\varphi(y^{\prime}%
)|}{|x^{\prime}-y^{\prime}|^{(1-\theta)\left(  n-1-\frac{1}{q_{1}}\right)  }%
}dy^{\prime}\right\Vert _{L^{q_{1}}([\Omega_{r}(x_{0})]^{\prime},dx^{\prime}%
)}, \label{eq:est_lineales_1.1_1}%
\end{align}
where the last inequality follows from (\ref{eq:P_kernel_est1}), with
$\theta\in(0,1)$ chosen so that $\theta\left(  n-1-\frac{1}{q_{1}}\right)
=\alpha$. With the definition of the Morrey norm in mind and using
(\ref{eq:est_lineales_1.1_1}), we arrive at
\begin{equation}
t^{\alpha}\left\Vert I_{1}[\varphi](\cdot,t)\right\Vert _{\mathcal{M}%
_{q_{1},\mu}(\mathbb{R}_{+}^{n})}\leq\tilde{c}_{n}\left\Vert \int
_{\mathbb{R}^{n-1}}\frac{1}{|x^{\prime}-y^{\prime}|^{n-1-\gamma_{1,1}}%
}|\varphi(y^{\prime})|dy^{\prime}\right\Vert _{\mathcal{M}_{q_{1},\mu
}(\mathbb{R}^{n-1})}, \label{aux-est-I1-1}%
\end{equation}
where
\[
\gamma_{1,1}=\frac{n-1-\mu}{\widetilde{q}_{0}}-\frac{n-1-\mu}{q_{1}}%
=\alpha+\frac{1}{q_{1}}\in(0,n-1).
\]
Applying estimate (\ref{xiao_liu_n}) to the R.H.S. of (\ref{aux-est-I1-1}) and
then taking the supremum over $t>0$ yields
\begin{equation}
\sup_{t>0}t^{\alpha}\left\Vert I_{1}[\varphi](\cdot,t)\right\Vert
_{\mathcal{M}_{q_{1},\mu}(\mathbb{R}_{+}^{n})}\leq C_{1,1}\left\Vert
\varphi\right\Vert _{\mathcal{M}_{\widetilde{q}_{0},\mu}(\mathbb{R}^{n-1})},
\label{eq:est_lineales_1.1}%
\end{equation}
for some constant $C_{1,1}>0$.

Next, using the kernel estimate (\ref{eq:P_kernel_est2}) with $\theta
(n-1)=\beta$ and $x_{n}=y_{n}=0,$ we obtain
\begin{align}
t^{\beta}\Vert I_{1}[\varphi](\cdot,0,t)\Vert_{\mathcal{M}_{q_{2},\mu
}(\mathbb{R}^{n-1})}  &  \leq c_{n}\frac{t^{\beta}}{t^{\theta(n-1)}}\left\Vert
\int_{\mathbb{R}^{n-1}}\frac{|\varphi(y^{\prime})|}{|x^{\prime}-y^{\prime
}|^{(1-\theta)(n-1)}}dy^{\prime}\right\Vert _{\mathcal{M}_{q_{2},\mu
}(\mathbb{R}^{n-1})}\nonumber\\
&  =c_{n}\left\Vert \int_{\mathbb{R}^{n-1}}\frac{|\varphi(y^{\prime}%
)|}{|x^{\prime}-y^{\prime}|^{n-1-\gamma_{1,2}}}dy^{\prime}\right\Vert
_{\mathcal{M}_{q_{2},\mu}(\mathbb{R}^{n-1})}\,, \label{aux-est-I1-2}%
\end{align}
where $\gamma_{1,2}=\frac{n-1-\mu}{\widetilde{q}_{0}}-\frac{n-1-\mu}{q_{2}%
}=\beta\in(0,n-1)$, since $p_{1}=2p_{2}-1$. Thus, in light of
(\ref{xiao_liu_n}), we can estimate the right-hand side of (\ref{aux-est-I1-2}%
) in order to get
\begin{equation}
\sup_{t>0}t^{\beta}\Vert I_{1}[\varphi](\cdot,0,t)\Vert_{\mathcal{M}%
_{q_{2},\mu}(\mathbb{R}^{n-1})}\leq C_{1,2}\Vert\varphi\Vert_{\mathcal{M}%
_{\widetilde{q}_{0},\mu}(\mathbb{R}^{n-1})}. \label{eq:est_lineales_1.2}%
\end{equation}
Now we turn to the remaining term in\ (\ref{norm1}). Proceeding as in
(\ref{eq:est_lineales_1.1_1}), but this time using (\ref{eq:P_kernel_est1})
with $\theta=0$, we have that%
\begin{align}
&  \Vert I_{1}[\varphi](\cdot,t)\Vert_{\mathcal{M}_{q_{0},\mu}(\mathbb{R}%
_{+}^{n})}\nonumber\\
&  \leq\left\Vert \int_{\mathbb{R}^{n-1}}\left\Vert P(x^{\prime}-y^{\prime
},x_{n}+t)\right\Vert _{L^{\frac{(n-\mu)(p_{1}-1)}{2}}((0,\infty),dx_{n}%
)}|\varphi(y^{\prime})|dy^{\prime}\right\Vert _{\mathcal{M}_{q _{0},\mu
}(\mathbb{R}^{n-1})}\nonumber\\
&  \leq\tilde{c}_{n}\left\Vert \int_{\mathbb{R}^{n-1}}\frac{1}{|x^{\prime
}-y^{\prime}|^{n-1-\gamma_{1,3}}}|\varphi(y^{\prime})|dy^{\prime}\right\Vert
_{\mathcal{M}_{q_{0},\mu}(\mathbb{R}^{n-1})}, \label{aux-est-I1-3}%
\end{align}
where $\gamma_{1,3}=\frac{n-1-\mu}{\widetilde{q}_{0}}-\frac{n-1-\mu}{q_{0}%
}=\frac{2}{(p_{1}-1)(n-\mu)}$. From (\ref{aux-est-I1-3}) and (\ref{xiao_liu_n}%
), we arrive at
\begin{equation}
\sup_{t>0}\Vert I_{1}[\varphi](\cdot,t)\Vert_{\mathcal{M}_{q_{0} ,\mu
}(\mathbb{R}_{+}^{n})}\leq C_{1,3}\Vert\varphi\Vert_{\mathcal{M}%
_{\widetilde{q}_{0},\mu}(\mathbb{R}^{n-1})}, \label{eq:est_lineales_1.3}%
\end{equation}
after taking the supremum over all $t>0$. Estimate (\ref{lineal_1}) then
follows by combining estimates (\ref{eq:est_lineales_1.1}),
(\ref{eq:est_lineales_1.2}) and (\ref{eq:est_lineales_1.3}), and setting
$C_{1}=C_{1,1}+C_{1,2}+C_{1,3}$.

\qed

In the following lemma we provide estimates for the operator $I_{2}.$

\begin{lemma}
\label{Lem-I2}Let $p_{1},p_{2}\in(1,\infty)$, $\mu\in\lbrack0,n-1)$, and let
$q_{0},q_{1},q_{2},\alpha,\beta$ be as in (\ref{Cond-H0}). Suppose further
that $l=q_{2}/p_{2}>1,$ $p_{1}=2p_{2}-1,$ and $0<\alpha<\beta p_{2}<1$. Then,
there exists a constant $C_{2}>0$ such that
\begin{equation}
\left\Vert I_{2}[g]\right\Vert _{\mathcal{X}}\leq C_{2}\left(  \sup
_{t>0}t^{\beta p_{2}}\left\Vert g(\cdot,t)\right\Vert _{\mathcal{M}_{l,\mu
}(\mathbb{R}^{n-1})}\right)  , \label{lineal_2}%
\end{equation}
for all $t^{\beta p_{2}}g\in L^{\infty}((0,\infty);\mathcal{M}_{l,\mu
}(\mathbb{R}^{n-1})).$
\end{lemma}

\textbf{Proof.} We begin by treating the component with $\mathcal{M}%
_{q_{1},\mu}(\mathbb{R}_{+}^{n})$ in the $\mathcal{X}$-norm. To this end, we
employ (\ref{eq:P_kernel_est1}) with $\theta\left(  n-1-\frac{1}{q_{1}%
}\right)  =\alpha-\beta p_{2}+1$ and proceed as in
(\ref{eq:est_lineales_1.1_1}) to obtain
\begin{align}
&  t^{\alpha}\left\Vert I_{2}[g](\cdot,t)\right\Vert _{\mathcal{M}_{q_{1},\mu
}(\mathbb{R}_{+}^{n})}\nonumber\\
&  \leq t^{\alpha}\left\Vert \int_{0}^{t}\int_{\mathbb{R}^{n-1}}\left\Vert
P(x^{\prime}-y^{\prime},x_{n}+t-s)\right\Vert _{L^{q_{1}}((0,\infty),dx_{n}%
)}\left\vert g(y^{\prime},s)\right\vert dy^{\prime}ds\right\Vert
_{\mathcal{M}_{q_{1},\mu}(\mathbb{R}^{n-1})}\nonumber\\
&  \leq\tilde{c}_{n}t^{\alpha}\int_{0}^{t}\frac{s^{-\beta p_{2}}%
}{(t-s)^{\theta\left(  n-1-\frac{1}{q_{1}}\right)  }}s^{\beta p_{2}}\left\Vert
\int_{\mathbb{R}^{n-1}}\frac{\left\vert g(y^{\prime},s)\right\vert
}{|x^{\prime}-y^{\prime}|^{n-1-\gamma_{2,1}}}dy^{\prime}\right\Vert
_{\mathcal{M}_{q_{1},\mu}(\mathbb{R}^{n-1})}ds\nonumber\\
&  =\tilde{c}_{n}t^{\alpha}\int_{0}^{t}\frac{s^{-\beta p_{2}}}{(t-s)^{\alpha
-\beta p_{2}+1}}s^{\beta p_{2}}\left\Vert \int_{\mathbb{R}^{n-1}}%
\frac{\left\vert g(y^{\prime},s)\right\vert }{|x^{\prime}-y^{\prime
}|^{n-1-\gamma_{2,1}}}dy^{\prime}\right\Vert _{\mathcal{M}_{q_{1},\mu
}(\mathbb{R}^{n-1})}ds, \label{eq:est_lineales_2.1_1}%
\end{align}
where
\[
\gamma_{2,1}=\frac{n-1-\mu}{l}-\frac{n-1-\mu}{q_{1}}=\alpha-\beta
p_{2}+1+\frac{1}{q_{1}}\in(0,n-1),
\]
and $l=q_{2}/p_{2}>1$. By applying estimates (\ref{eq:est_lineales_2.1_1}) and
(\ref{xiao_liu_n}), and making a change of variables, we arrive at
\begin{align}
&  t^{\alpha}\left\Vert I_{2}[g](\cdot,t)\right\Vert _{\mathcal{M}_{q_{1},\mu
}(\mathbb{R}_{+}^{n})}\nonumber\\
&  \leq\bar{C}_{2,1}t^{\alpha}\int_{0}^{t}\frac{s^{-\beta p_{2}}%
}{(t-s)^{\alpha-\beta p_{2}+1}}s^{\beta p_{2}}\left\Vert g(\cdot,s)\right\Vert
_{\mathcal{M}_{l,\mu}(\mathbb{R}^{n-1})}ds\label{c_lineares_2_1}\\
&  \leq\bar{C}_{2,1}\int_{0}^{1}(1-r)^{-(\alpha-\beta p_{2}+1)}r^{-\beta
p_{2}}dr\left(  \sup_{t>0}t^{\beta p_{2}}\left\Vert g(\cdot,t)\right\Vert
_{\mathcal{M}_{l,\mu}(\mathbb{R}^{n-1})}\right)  ,\nonumber
\end{align}
where the beta integral is finite due to the conditions $\alpha<\beta p_{2}$
and $\beta p_{2}<1$. Consequently,
\begin{equation}
\sup_{t>0}t^{\alpha}\left\Vert I_{2}[g](\cdot,t)\right\Vert _{\mathcal{M}%
_{q_{1},\mu}(\mathbb{R}_{+}^{n})}\leq C_{2,1}\left(  \sup_{t>0}t^{\beta p_{2}%
}\left\Vert g(\cdot,t)\right\Vert _{\mathcal{M}_{l,\mu}(\mathbb{R}^{n-1}%
)}\right)  . \label{eq:est_lineales_2.1}%
\end{equation}
Next, we turn to the component with $\mathcal{M}_{q_{2},\mu}(\mathbb{R}%
^{n-1})$ in the $\mathcal{X}$-norm. Taking $\theta$ in (\ref{eq:P_kernel_est2}%
) such that $\theta(n-1)=1-\beta(p_{2}-1)$, it follows that
\begin{align}
&  t^{\beta}\left\Vert I_{2}[g](\cdot,0,t)\right\Vert _{\mathcal{M}_{q_{2}%
,\mu}(\mathbb{R}^{n-1})}\nonumber\\
&  \leq c_{n}t^{\beta}\int_{0}^{t}\frac{s^{-\beta p_{2}}}{(t-s)^{\theta(n-1)}%
}s^{\beta p_{2}}\left\Vert \int_{\mathbb{R}^{n-1}}\frac{\left\vert
g(y^{\prime},s)\right\vert }{|x^{\prime}-y^{\prime}|^{n-1-\gamma_{2,2}}%
}dy^{\prime}\right\Vert _{\mathcal{M}_{q_{2},\mu}(\mathbb{R}^{n-1}%
)}ds\nonumber\\
&  \leq\bar{C}_{2,2}t^{\beta}\int_{0}^{t}\frac{s^{-\beta p_{2}}}%
{(t-s)^{1-\beta(p_{2}-1)}}s^{\beta p_{2}}\left\Vert g(\cdot,s)\right\Vert
_{\mathcal{M}_{l,\mu}(\mathbb{R}^{n-1})}ds. \label{c_lineares_2_2}%
\end{align}
In the last step, we have used estimate (\ref{xiao_liu_n}) with $l=\frac
{q_{2}}{p_{2}}$ and the parameter $\gamma_{2,2}=\frac{n-1-\mu}{l}%
-\frac{n-1-\mu}{q_{2}}=1-\beta(p_{2}-1)\in(0,n-1)$. Now, changing variables
via $s=tr$ and using the conditions $\beta p_{2}<1$ and $\beta(p_{2}-1)>0,$ we
obtain%
\begin{align}
\sup_{t>0}t^{\beta}\left\Vert I_{2}[g](\cdot,0,t)\right\Vert _{\mathcal{M}%
_{q_{2},\mu}(\mathbb{R}^{n-1})}  &  \leq\bar{C}_{2,2}\int_{0}^{1}r^{-\beta
p_{2}}(1-r)^{-(1-\beta(p_{2}-1))}dr\left(  \sup_{t>0}t^{\beta p_{2}}\left\Vert
g(\cdot,t)\right\Vert _{\mathcal{M}_{l,\mu}(\mathbb{R}^{n-1})}\right)
\nonumber\\
&  \leq C_{2,2}\left(  \sup_{t>0}t^{\beta p_{2}}\left\Vert g(\cdot
,t)\right\Vert _{\mathcal{M}_{l,\mu}(\mathbb{R}^{n-1})}\right)  .
\label{eq:est_lineales_2.2}%
\end{align}
Finally, we handle the term of the $\mathcal{X}$-norm involving the space
$\mathcal{M}_{q_{0},\mu}(\mathbb{R}_{+}^{n}).$ Here we apply
(\ref{eq:P_kernel_est1}) with a parameter $\theta$ chosen such that
$\theta\left(  n-1-\frac{1}{q_{0}}\right)  =1-\beta p_{2}$. Recalling that
$l=\frac{q_{2}}{p_{2}}$, this leads to the following estimate%
\begin{align}
&  \left\Vert I_{2}[g](\cdot,t)\right\Vert _{\mathcal{M}_{q_{0} ,\mu
}(\mathbb{R}_{+}^{n})}\nonumber\\
&  \leq\left\Vert \int_{0}^{t}\int_{\mathbb{R}^{n-1}}\left\Vert P(x^{\prime
}-y^{\prime},x_{n}+t-s)\right\Vert _{L^{\frac{(n-\mu)(p_{1}-1)}{2}}%
((0,\infty),dx_{n})}\left\vert g(y^{\prime},s)\right\vert dy^{\prime
}ds\right\Vert _{\mathcal{M}_{q_{0},\mu}(\mathbb{R}^{n-1})}\nonumber\\
&  \leq\tilde{c}_{n}\int_{0}^{t}\frac{s^{-\beta p_{2}}}{(t-s)^{1-\beta p_{2}}%
}s^{\beta p_{2}}\left\Vert \int_{\mathbb{R}^{n-1}}\frac{\left\vert
g(y^{\prime},s)\right\vert }{|x^{\prime}-y^{\prime}|^{n-1-\gamma_{2,3}}%
}dy^{\prime}\right\Vert _{\mathcal{M}_{q_{0},\mu}(\mathbb{R}^{n-1}%
)}ds\nonumber\\
&  \leq\bar{C}_{2,3}\int_{0}^{t}\frac{s^{-\beta p_{2}}}{(t-s)^{1-\beta p_{2}}%
}s^{\beta p_{2}}\left\Vert g(\cdot,s)\right\Vert _{\mathcal{M}_{l,\mu
}(\mathbb{R}^{n-1})}ds. \label{c_lineares_2_3}%
\end{align}
In the above estimate we have used (\ref{xiao_liu_n}) with%
\[
\gamma_{2,3}=\frac{n-1-\mu}{l}-\frac{n-1-\mu}{q_{0}}=1-\beta p_{2}+\frac
{1}{q_{0}}\in(0,n-1),
\]
which can be verified using the relation $p_{1}=2p_{2}-1$. Making the change
of variables $s=tr$ in (\ref{c_lineares_2_3}) and using the condition $0<\beta
p_{2}<1$ yields%
\begin{align}
\sup_{t>0}\left\Vert I_{2}[g](\cdot,t)\right\Vert _{\mathcal{M}_{q_{0} ,\mu
}(\mathbb{R}_{+}^{n})}  &  \leq\bar{C}_{2,3}\int_{0}^{1}r^{-\beta p_{2}%
}(1-r)^{-(1-\beta p_{2})}dr\left(  \sup_{t>0}t^{\beta p_{2}}\left\Vert
g(\cdot,t)\right\Vert _{\mathcal{M}_{l,\mu}(\mathbb{R}^{n-1})}\right)
\nonumber\\
&  \leq C_{2,3}\left(  \sup_{t>0}t^{\beta p_{2}}\left\Vert g(\cdot
,t)\right\Vert _{\mathcal{M}_{l,\mu}(\mathbb{R}^{n-1})}\right)  .
\label{eq:est_lineales_2.3}%
\end{align}
Combining (\ref{eq:est_lineales_2.1}), (\ref{eq:est_lineales_2.2}), and
(\ref{eq:est_lineales_2.3}), the resulting estimate is (\ref{lineal_2}) with
$C_{2}=C_{2,1}+C_{2,2}+C_{2,3}$.\qed

The next lemma concerns estimates for the operator $I_{3}.$

\begin{lemma}
\label{Lem-I3}Let $p_{1},p_{2}\in(1,\infty)$, $\mu\in\lbrack0,n-1)$, and let
$q_{0},q_{1},q_{2},\alpha,\beta$ be as in (\ref{Cond-H0}). Suppose further
that $l=q_{1}/p_{1}>1,$ $p_{1}=2p_{2}-1$, and $0<\beta<\alpha p_{1}<1$. Then,
there exists a constant $C_{3}>0$ such that%
\begin{equation}
\left\Vert I_{3}[f]\right\Vert _{\mathcal{X}}\leq C_{3}\left(  \sup
_{t>0}t^{\alpha p_{1}}\left\Vert f(\cdot,t)\right\Vert _{\mathcal{M}_{l,\mu
}(\mathbb{R}_{+}^{n})}\right)  \label{lineal_3}%
\end{equation}
for all $t^{\alpha p_{1}}f\in L^{\infty}((0,\infty);\mathcal{M}_{l,\mu
}(\mathbb{R}_{+}^{n})).$
\end{lemma}

\textbf{Proof.} We consider (\ref{eq:P_kernel_est2}) by choosing $\theta$ so
that $\theta(n-1)=1-\alpha(p_{1}-1)$, use the extension (\ref{extension1}) of
$f$, and then employ (\ref{xiao_liu_n}) to estimate $I_{3}$ as follows%
\begin{align}
t^{\alpha}\left\Vert I_{3}[f](\cdot,t)\right\Vert _{\mathcal{M}_{q_{1},\mu
}(\mathbb{R}_{+}^{n})}  &  \leq c_{n}t^{\alpha}\int_{0}^{t}\frac{s^{-\alpha
p_{1}}}{(t-s)^{\theta(n-1)}}s^{\alpha p_{1}}\left\Vert \int_{\mathbb{R}%
_{+}^{n}}\frac{\left\vert f(y,s)\right\vert }{|x-y|^{n-\gamma_{3,1}}%
}dy\right\Vert _{\mathcal{M}_{q_{1},\mu}(\mathbb{R}_{+}^{n})}ds\nonumber\\
&  \leq c_{n}t^{\alpha}\int_{0}^{t}\frac{s^{-\alpha p_{1}}}{(t-s)^{1-\alpha
(p_{1}-1)}}s^{\alpha p_{1}}\left\Vert \int_{\mathbb{R}^{n}}\frac{\left\vert
\tilde{f}(y,s)\right\vert }{|x-y|^{n-\gamma_{3,1}}}dy\right\Vert
_{\mathcal{M}_{q_{1},\mu}(\mathbb{R}^{n})}ds\nonumber\\
&  \leq C\int_{0}^{1}(1-r)^{-(1-\alpha(p_{1}-1))}r^{-\alpha p_{1}}dr\left(
\sup_{t>0}t^{\alpha p_{1}}\left\Vert \tilde{f}(\cdot,t)\right\Vert
_{\mathcal{M}_{l,\mu}(\mathbb{R}^{n})}\right) \label{c_lineares_3_1}\\
&  \leq C\left(  \sup_{t>0}t^{\alpha p_{1}}\left\Vert \tilde{f}(\cdot
,t)\right\Vert _{\mathcal{M}_{l,\mu}(\mathbb{R}^{n})}\right)  ,\nonumber
\end{align}
where above we have that $\gamma_{3,1}=\frac{n-\mu}{l}-\frac{n-\mu}{q_{1}%
}=2-\alpha(p_{1}-1)\in(0,n)$ and used the change of variables $s=tr$ and
conditions $\alpha p_{1}<1$ and $\alpha(p_{1}-1)>0$. Recalling the estimate
for extension (\ref{cont-ext1}) implies that
\begin{equation}
\sup_{t>0}t^{\alpha}\left\Vert I_{3}[f](\cdot,t)\right\Vert _{\mathcal{M}%
_{q_{1},\mu}(\mathbb{R}_{+}^{n})}\leq C_{3,1}\left(  \sup_{t>0}t^{\alpha
p_{1}}\left\Vert f(\cdot,t)\right\Vert _{\mathcal{M}_{l,\mu}(\mathbb{R}%
_{+}^{n})}\right)  . \label{eq:est_lineales_3.1}%
\end{equation}

In the sequel, we focus on the part of the $\mathcal{X}$-norm involving
$\mathcal{M}_{q_{2},\mu}(\mathbb{R}^{n-1})$. To this end, we choose $\theta$
in (\ref{eq:P_kernel_est2}) satisfying $\theta(n-1)=\beta-\alpha p_{1}+1$, and
proceed as follows%
\begin{align}
&  \sup_{t>0}t^{\beta}\left\Vert I_{3}[f](\cdot,0,t)\right\Vert _{\mathcal{M}%
_{q_{2},\mu}(\mathbb{R}^{n-1})}\nonumber\\
&  \leq c_{n}\sup_{t>0}t^{\beta}\int_{0}^{t}\frac{s^{-\alpha p_{1}}%
}{(t-s)^{\theta(n-1)}}s^{\alpha p_{1}}\left\Vert \int_{\mathbb{R}_{+}^{n}%
}\frac{\left\vert f(y,s)\right\vert }{|(x^{\prime},0)-y|^{n-\gamma_{3,2}}%
}dy\right\Vert _{\mathcal{M}_{q_{2},\mu}(\mathbb{R}^{n-1})}ds\nonumber\\
&  \leq c_{n}\sup_{t>0}t^{\beta}\int_{0}^{t}\frac{s^{-\alpha p_{1}}%
}{(t-s)^{1-\alpha p_{1}+\beta}}\left\Vert \int_{\mathbb{R}^{n}}\frac
{\left\vert \tilde{f}(y,s)\right\vert }{|(x^{\prime},0)-y|^{n-\gamma_{3,2}}%
}dy\right\Vert _{\mathcal{M}_{q_{2},\mu}(\mathbb{R}^{n-1})}ds\nonumber\\
&  \leq C\int_{0}^{1}r^{-\alpha p_{1}}(1-r)^{-(1-\alpha p_{1}+\beta)}dr\left(
\sup_{t>0}t^{\alpha p_{1}}\left\Vert \tilde{f}(\cdot,t)\right\Vert
_{\mathcal{M}_{l,\mu}(\mathbb{R}^{n})}\right) \nonumber\\
&  \leq C_{3,2}\left(  \sup_{t>0}t^{\alpha p_{1}}\left\Vert f(\cdot
,t)\right\Vert _{\mathcal{M}_{l,\mu}(\mathbb{R}_{+}^{n})}\right)  ,
\label{eq:est_lineales_3.2}%
\end{align}
where $\gamma_{3,2}=\frac{n-\mu}{l}-\frac{n-1-\mu}{q_{2}}=2-\alpha p_{1}%
+\beta\in(0,n-1)$, since $\beta<\alpha p_{1}$. In the above estimates, we have
used (\ref{xiao_liu_n_1}), the conditions $\beta<\alpha p_{1}$ and $\alpha
p_{1}<1$ (ensuring convergence of the Beta integral), and finally
(\ref{cont-ext1}).

Now we treat the component with $\mathcal{M}_{q_{0},\mu}(\mathbb{R}^{n-1})$ in
the $\mathcal{X}$-norm. Taking $\theta$ in (\ref{eq:P_kernel_est2}) so that
$\theta(n-1)=1-\alpha p_{1}$, we can estimate
\begin{align}
&  \left\Vert I_{3}[f](\cdot,t)\right\Vert _{\mathcal{M}_{q_{0} ,\mu
}(\mathbb{R}_{+}^{n})}\nonumber\\
&  \leq c_{n}\int_{0}^{t}\frac{s^{-\alpha p_{1}}}{(t-s)^{\theta\left(
n-1\right)  }}s^{\alpha p_{1}}\left\Vert \int_{\mathbb{R}^{n}}\frac{\tilde
{f}(y,s)}{|x-y|^{n-\gamma_{3,3}}}dy\right\Vert _{\mathcal{M}_{q_{0} ,\mu
}(\mathbb{R}^{n})}ds\nonumber\\
&  \leq C\int_{0}^{t}\frac{s^{-\alpha p_{1}}}{(t-s)^{1-\alpha p_{1}}}s^{\alpha
p_{1}}\left\Vert \tilde{f}(\cdot,s)\right\Vert _{\mathcal{M}_{l,\mu
}(\mathbb{R}^{n})}ds \label{c_lineares_3_3}%
\end{align}
where in the last step we have used (\ref{xiao_liu_n}) with $\gamma
_{3,3}=\frac{n-\mu}{l}-\frac{n-\mu}{q_{0}}=2-\alpha p_{1}\in(0,n)$. Making the
change $s=tr$ and using the conditions $\alpha p_{1}>0$ and $\alpha p_{1}<1$
in (\ref{c_lineares_3_3}), we arrive at
\begin{align}
\sup_{t>0}\left\Vert I_{3}[f](\cdot,t)\right\Vert _{\mathcal{M}_{q_{0} ,\mu
}(\mathbb{R}_{+}^{n})}  &  \leq C\int_{0}^{1}r^{-\alpha p_{1}}%
(1-r)^{-(1-\alpha p_{1})}dr\left(  \sup_{t>0}t^{\alpha p_{1}}\left\Vert
\tilde{f}(\cdot,t)\right\Vert _{\mathcal{M}_{l,\mu}(\mathbb{R}^{n})}\right)
\nonumber\\
&  \leq C_{3,3}\left(  \sup_{t>0}t^{\alpha p_{1}}\left\Vert f(\cdot
,t)\right\Vert _{\mathcal{M}_{l,\mu}(\mathbb{R}_{+}^{n})}\right)  .
\label{eq:est_lineales_3.3}%
\end{align}
Considering $C_{3}=C_{3,1}+C_{3,2}+C_{3,3}$ and putting together estimates
(\ref{eq:est_lineales_3.1}), (\ref{eq:est_lineales_3.2}) and
(\ref{eq:est_lineales_3.3}), we conclude (\ref{lineal_3}).\qed

In what follows, we proceed to derive the estimates associated with $I_{4}$.

\begin{lemma}
\label{Lem-I4}Let $p_{1},p_{2}\in(1,\infty)$, $\mu\in\lbrack0,n-1)$, and let
$q_{0},q_{1},q_{2},\alpha,\beta$ be as in (\ref{Cond-H0}). In addition, assume
that $l_{1}=\frac{(n-\mu)q_{1}}{n-\mu+2q_{1}}>1$, $l_{2}=q_{0}/p_{1}>1$, and
$\alpha,\beta>0$. Then, there exists a constant $C_{4}>0$ such that
\begin{equation}
\left\Vert I_{4}[f]\right\Vert _{\mathcal{X}}\leq C_{4}\left(  \sup
_{t>0}t^{\alpha}\left\Vert f(\cdot,t)\right\Vert _{\mathcal{M}_{l_{1},\mu
}(\mathbb{R}_{+}^{n})}+\sup_{t>0}\left\Vert f(\cdot,t)\right\Vert
_{\mathcal{M}_{l_{2},\mu}(\mathbb{R}_{+}^{n})}\right)  \label{lineal_4}%
\end{equation}
for all $f\in L^{\infty}((0,\infty);\mathcal{M}_{l_{2},\mu}(\mathbb{R}_{+}%
^{n}))$ with $t^{\alpha}f\in L^{\infty}((0,\infty);\mathcal{M}_{l_{1},\mu
}(\mathbb{R}_{+}^{n})).$
\end{lemma}

\textbf{Proof.} By (\ref{eq:Green_est}) and (\ref{extension1}), we have
\begin{align}
t^{\alpha}\left\Vert I_{4}[f](\cdot,t)\right\Vert _{\mathcal{M}_{q_{1},\mu
}(\mathbb{R}_{+}^{n})}  &  \leq\bar{c}_{n}t^{\alpha}\left\Vert \int
_{\mathbb{R}_{+}^{n}}\frac{\left\vert f(y,t)\right\vert }{|x-y|^{n-2}%
}dy\right\Vert _{\mathcal{M}_{q_{1},\mu}(\mathbb{R}_{+}^{n})}\nonumber\\
&  \leq\bar{c}_{n}t^{\alpha}\left\Vert \int_{\mathbb{R}^{n}}\frac{\left\vert
\tilde{f}(y,t)\right\vert }{|x-y|^{n-2}}dy\right\Vert _{\mathcal{M}_{q_{1}%
,\mu}(\mathbb{R}^{n})}. \label{c_lineares_4_1}%
\end{align}
Observing that $2=\frac{n-\mu}{l_{1}}-\frac{n-\mu}{q_{1}}$, and invoking
(\ref{xiao_liu_n}) followed by (\ref{cont-ext1}), it follows from
(\ref{c_lineares_4_1}) that
\begin{align}
\sup_{t>0}t^{\alpha}\left\Vert I_{4}[f]\right\Vert _{\mathcal{M}_{q_{1},\mu
}(\mathbb{R}_{+}^{n})}  &  \leq C\sup_{t>0}t^{\alpha}\left\Vert \tilde
{f}(\cdot,t)\right\Vert _{\mathcal{M}_{l_{1},\mu}(\mathbb{R}^{n})}\nonumber\\
&  \leq C_{4,1}\left(  \sup_{t>0}t^{\alpha}\left\Vert f(\cdot,t)\right\Vert
_{\mathcal{M}_{l_{1},\mu}(\mathbb{R}_{+}^{n})}\right)  .
\label{eq:est_lineales_4.1}%
\end{align}
In turn, noting that $G(x^{\prime},0,y)=0$ for all $y\in\mathbb{R}_{+}^{n}$,
we arrive at%
\[
I_{4}[f](x^{\prime},0,t)=\int_{\mathbb{R}_{+}^{n}}G(x^{\prime},0,y)\left[
f(y,t)\right]  dy=0.
\]
Consequently,
\begin{equation}
\sup_{t>0}t^{\beta}\left\Vert I_{4}[f](\cdot,0,t)\right\Vert _{\mathcal{M}%
_{q_{2},\mu}(\mathbb{R}^{n-1})}=0. \label{eq:est_lineales_4.2}%
\end{equation}
Finally, for the term involving the space $\mathcal{M}_{q_{0} ,\mu}%
(\mathbb{R}_{+}^{n})$, we employ (\ref{eq:Green_est}) together with
(\ref{xiao_liu_n}) to obtain
\begin{align}
\sup_{t>0}\left\Vert I_{4}[f](\cdot,t)\right\Vert _{\mathcal{M}_{q_{0} ,\mu
}(\mathbb{R}_{+}^{n})}  &  \leq\bar{c}_{n}\sup_{t>0}\left\Vert \int
_{\mathbb{R}^{n}}\frac{\left\vert \tilde{f}(y,t)\right\vert }{|x-y|^{n-2}%
}dy\right\Vert _{\mathcal{M}_{q_{0} ,\mu}(\mathbb{R}^{n})}\nonumber\\
&  \leq C\sup_{t>0}\left\Vert \tilde{f}(\cdot,t)\right\Vert _{\mathcal{M}%
_{l_{2},\mu}(\mathbb{R}^{n})}\nonumber\\
&  \leq C_{4,2}\left(  \sup_{t>0}\left\Vert f(\cdot,t)\right\Vert
_{\mathcal{M}_{l_{2},\mu}(\mathbb{R}_{+}^{n})}\right)  .
\label{eq:est_lineales_4.3}%
\end{align}
Here we have used the relation $\frac{n-\mu}{l_{2}}-\frac{n-\mu}{q_{0}}=2.$

In view of (\ref{eq:est_lineales_4.1}), (\ref{eq:est_lineales_4.2}), and
(\ref{eq:est_lineales_4.3}), we obtain (\ref{lineal_4}) by setting
$C_{4}=C_{4,1}+C_{4,2}$.

\qed

The following lemma concerns the Poisson kernel in the setting of block
spaces. In particular, we study a convergence property showing that this
kernel serves as an approximate identity acting on blocks in these spaces.
Since we were unable to find these results in the existing literature, we
include the lemma below together with its proof for the reader's convenience.

\begin{lemma}
\label{Lem-Poisson-Block}Let $1<q<\infty$, $0\leq\mu<n-1$, and let $A$ be a
$(q^{\prime},\mu)$-block in $\mathbb{R}^{n-1}$. Denoting by $P_{t}(x^{\prime
})=P(x^{\prime},t)$ the Poisson kernel defined in (\ref{P-expression-1}), we
have that
\begin{equation}
\lim_{t\rightarrow0^{+}}\Vert P_{t}\ast A-A\Vert_{\mathcal{H}_{q^{\prime},\mu
}(\mathbb{R}^{n-1})}=0. \label{lem:Poiss_aprox_resultado1}%
\end{equation}

Moreover, let $1<p<q<\infty$ satisfy $\gamma_{1}=\frac{n-1-\mu}{p}%
-\frac{n-1-\mu}{q}<n-1$, and consider $v_{0}\in\mathcal{M}_{p,\mu}%
(\mathbb{R}^{n-1})$ and $\psi\in\mathcal{H}_{q^{\prime},\mu}(\mathbb{R}%
^{n-1})$ with block decomposition $\psi=\sum_{i=1}^{\infty}\lambda_{i}A_{i}$.
Then, for any $N\in\mathbb{N},$ we obtain the estimate
\begin{equation}
|\langle P_{t}\ast v_{0},\psi\rangle|\leq C_{0}t^{-\gamma_{1}}\Vert v_{0}%
\Vert_{\mathcal{M}_{p,\mu}(\mathbb{R}^{n-1})}\sum_{i=N+1}^{\infty}|\lambda
_{i}|+(C_{N}+D_{N}t)\Vert v_{0}\Vert_{\mathcal{M}_{p,\mu}(\mathbb{R}^{n-1})},
\label{lem:Poiss_aprox_resultado3-1}%
\end{equation}
where the constant $C_{0}>0$ is independent of $N$, while $C_{N},D_{N}>0$
depend on $N$.

Furthermore, let $1<p<q<\infty$ such that $\frac{n-\mu}{p}>1$ and
$1<\gamma_{2}=\frac{n-\mu}{p}-\frac{n-1-\mu}{q}<n$, and consider
$v\in\mathcal{M}_{p,\mu}(\mathbb{R}_{+}^{n})$ and $\psi\in\mathcal{H}%
_{q^{\prime},\mu}(\mathbb{R}^{n-1})$ with block decomposition $\psi=\sum
_{i=1}^{\infty}\lambda_{i}A_{i}$. Then, for any $N\in\mathbb{N},$ we have
\begin{equation}
\left\langle \int_{0}^{\infty}P_{y_{n}+t}\ast v(\cdot,y_{n})dy_{n}%
,\psi\right\rangle \leq\tilde{C}_{0}t^{1-\gamma_{2}}\Vert v\Vert
_{\mathcal{M}_{p,\mu}(\mathbb{R}_{+}^{n})}\sum_{i=N+1}^{\infty}|\lambda
_{i}|+(\tilde{C}_{N}+\tilde{D}_{N}t)\Vert v\Vert_{\mathcal{M}_{p,\mu
}(\mathbb{R}_{+}^{n})}, \label{lem:Poiss_aprox_resultado2}%
\end{equation}
where the constant $\tilde{C}_{0}>0$ is independent of $N$, while $\tilde
{C}_{N},\tilde{D}_{N}>0$ depend on $N$.
\end{lemma}

\textbf{Proof.} Let $\rho_{\varepsilon}\in C_{c}^{\infty}(\mathbb{R}^{n-1})$
be a mollifier, where $\varepsilon>0$ will be chosen later. We begin by
deriving a suitable estimate for $\Vert P_{t}\ast(\rho_{\varepsilon}\ast
A_{i})\Vert_{\mathcal{H}_{p^{\prime},\mu}}$ with $p\leq q$. Note that the
regularization provides bounded $L^{a}$-norms independently of $t$. If $A_{i}$
is a $(q^{\prime},\mu)$-block supported in $B(x_{i}^{\prime},R_{i}%
)\subset\mathbb{R}^{n-1}$, then%
\begin{equation}
\Vert A_{i}\Vert_{L^{q^{\prime}}(\mathbb{R}^{n-1})}\leq R_{i}^{-\mu/q}.
\label{lem:Poiss_aprox_1}%
\end{equation}
We decompose $\mathbb{R}^{n-1}$ into the regions $B_{1,i}=B(x_{i}^{\prime
},2R_{i})$ and $B_{k,i}=B(x_{i}^{\prime},2^{k}R_{i})\backslash B(x_{i}%
^{\prime},2^{k-1}R_{i})$, for $k\geq2$, and consider the functions
\[
a_{k,i}:=\alpha_{k,i}^{-1}\chi_{B_{k,i}}\cdot(P_{t}\ast(\rho_{\varepsilon}\ast
A_{i})),
\]
where $\alpha_{k,i}$ are coefficients to be chosen appropriately.

For $k=1$, we can choose $a\geq1$ such that $1+\frac{1}{p^{\prime}}=\frac
{1}{a}+\frac{1}{q^{\prime}}$. Using Young's inequality together with
(\ref{lem:Poiss_aprox_1}), we obtain
\begin{align}
\Vert a_{1,i}\Vert_{L^{p^{\prime}}}  &  \leq\alpha_{1,i}^{-1}\Vert P_{t}%
\ast(\rho_{\varepsilon}\ast A_{i})\Vert_{L^{p^{\prime}}(B_{1,i})}\nonumber\\
&  \leq\alpha_{1,i}^{-1}\Vert\rho_{\varepsilon}\Vert_{L^{a}}\Vert A_{i}%
\Vert_{L^{q^{\prime}}}\nonumber\\
&  \leq\alpha_{1,i}^{-1}2^{\mu/p}R_{i}^{\mu\left(  \frac{1}{p}-\frac{1}%
{q}\right)  }\Vert\rho_{\varepsilon}\Vert_{L^{a}}(2R_{i})^{-\mu/p}.
\label{lem:Poiss_cuenta_1}%
\end{align}
As $\Vert\rho_{\varepsilon}\Vert_{L^{a}}<\infty$, we can choose $\alpha
_{1,i}=2^{\mu/p}R_{i}^{\mu\left(  \frac{1}{p}-\frac{1}{q}\right)  }\Vert
\rho_{\varepsilon}\Vert_{L^{a}}$ so that $\Vert a_{1,i}\Vert_{L^{p^{\prime}}%
}\leq(2R_{i})^{-\mu/p}$.

Next let $k\geq2$. For each $y^{\prime}\in B_{k,i}$ and $x^{\prime}%
\in\mathrm{supp}(\rho_{\varepsilon}\ast A_{i})$, with $\varepsilon$ chosen
sufficiently small (e.g., $\varepsilon<R_{i}/2$), we have $|x^{\prime
}-y^{\prime}|\geq2^{k-2}R_{i}/2$ for all $k\geq2$. Using this estimate
together with Young's and H\"{o}lder's inequalities, and
(\ref{lem:Poiss_aprox_1}), we arrive at
\begin{align*}
|(P_{t}\ast(\rho_{\varepsilon}\ast A_{i}))(y^{\prime})|  &  \leq2^{n}%
c_{n}\frac{t}{(2^{k-2}R_{i})^{n}}\Vert\rho_{\varepsilon}\Vert_{L^{1}}\Vert
A_{i}\Vert_{L^{q^{\prime}}}|B(x_{i},R_{i})|^{1/q}\\
&  \leq C\frac{t}{2^{k}R_{i}}2^{-k(n-1)}R_{i}^{-(n-1)/q^{\prime}-\mu/q}%
,\quad\forall y^{\prime}\in B_{k,i},
\end{align*}
and then%
\begin{align*}
\Vert a_{k,i}\Vert_{L^{p^{\prime}}}  &  \leq\alpha_{k,i}^{-1}\sup_{y^{\prime
}\in B_{k,i}}|(P_{t}\ast(\rho_{\varepsilon}\ast A_{i}))(y^{\prime}%
)|\cdot|B_{k,i}|^{1/p^{\prime}}\\
&  \leq\alpha_{k,i}^{-1}C\frac{t}{2^{k}R_{i}}2^{-k\left(  \frac{n-1-\mu}%
{p}\right)  }R_{i}^{-(n-1-\mu)\left(  \frac{1}{p}-\frac{1}{q}\right)  }%
(2^{k}R_{i})^{-\mu/p},\quad\forall y^{\prime}\in B_{k,i}.
\end{align*}
Taking
\[
\alpha_{k,i}=C\frac{t}{R_{i}}\left(  2^{1+\frac{n-1-\mu}{p}}\right)
^{-k}R_{i}^{-(n-1-\mu)\left(  \frac{1}{p}-\frac{1}{q}\right)  }\text{, for
}k\geq2,
\]
it follows that
\[
\Vert a_{k,i}\Vert_{L^{p^{\prime}}}\leq(2^{k}R_{i})^{-\mu/p}.
\]
Thus, by the definition of block space $\mathcal{H}_{p^{\prime},\mu}$ (see
(\ref{def_blocks_space})), we conclude that
\begin{equation}
\Vert P_{t}\ast(\rho_{\varepsilon}\ast A_{i})\Vert_{\mathcal{H}_{p^{\prime
},\mu}(\mathbb{R}^{n-1})}\leq\sum_{k=1}^{\infty}\alpha_{k,i}=2^{\mu/p}%
R_{i}^{\mu\left(  \frac{1}{p}-\frac{1}{q}\right)  }\Vert\rho_{\varepsilon
}\Vert_{L^{a}}+C\frac{t}{R_{i}}R_{i}^{-(n-1-\mu)\left(  \frac{1}{p}-\frac
{1}{q}\right)  }\sum_{k=2}^{\infty}\left(  2^{1+\frac{n-1-\mu}{p}}\right)
^{-k}<\infty, \label{lem:Poisson_finitud}%
\end{equation}
because $1+\frac{n-1-\mu}{p}>0$.

Now we turn to the convergence (\ref{lem:Poiss_aprox_resultado1}). Assume that
$A$ is a $(q^{\prime},\mu)$-block supported in $B(x,R)$. Since $A\in
L^{q^{\prime}}(\mathbb{R}^{n-1}),$ the Poisson kernel provides an
approximation of the identity in $L^{q^{\prime}}$ (see, e.g., \cite[p.~242]%
{folland1999real}), hence
\[
\lim_{t\rightarrow0^{+}}\Vert P_{t}\ast A-A\Vert_{L^{q^{\prime}}}=0.
\]
For $\eta>0$, then there exists $\delta>0$ such that
\[
\Vert P_{t}\ast A-A\Vert_{L^{q^{\prime}}}\leq\frac{\eta}{2}(2R)^{-\mu
/q},\text{ for all }0<t<\delta.
\]
We decompose $\mathbb{R}^{n-1}$ as
\[
B_{1}=B(x^{\prime},2R)\text{ and }B_{k}=B(x^{\prime},2^{k}R)\setminus
B(x^{\prime},2^{k-1}R),\text{ for }k\geq2.
\]
For $k=1,$ define
\[
A_{1}=\frac{1}{\alpha_{1}}\chi_{B_{1}}\cdot(P_{t}\ast A-A)\text{ with }%
\alpha_{1}=\frac{\eta}{2}\text{.}%
\]
For $k\geq2$, set
\[
A_{k}=\alpha_{k}^{-1}\chi_{B_{k}}\cdot(P_{t}\ast A-A)\text{ with }\alpha
_{k}=C\frac{t}{R}\left(  2^{1+\frac{n-1-\mu}{q}}\right)  ^{-k}.
\]
Note that, for $y^{\prime}\in B_{k},$ we have $(P_{t}\ast A-A)(y^{\prime
})=(P_{t}\ast A)(y^{\prime})$ since $A$ vanishes outside $B(x^{\prime},R)$.
Arguing as in the estimate used in the proof of (\ref{lem:Poisson_finitud})
with $p=q$, one verifies that
\[
\left\Vert A_{k}\right\Vert _{L^{q^{\prime}}}\leq(2^{k}R)^{-\mu/q}.
\]
Hence, by the definition of the block norm, we obtain
\[
\Vert P_{t}\ast A-A\Vert_{\mathcal{H}_{q^{\prime},\mu}(\mathbb{R}^{n-1})}%
\leq\alpha_{1}+\sum_{k=2}^{\infty}\alpha_{k}=\frac{\eta}{2}+C\frac{t}{R}%
\sum_{k=2}^{\infty}\left(  2^{1+\frac{n-1-\mu}{q}}\right)  ^{-k}.
\]
Then, by choosing a sufficiently small $\delta>0$ (if necessary), we can
ensure that
\[
\Vert P_{t}\ast A-A\Vert_{\mathcal{H}_{q^{\prime},\mu}(\mathbb{R}^{n-1})}%
<\eta,\text{ for all }0<t<\delta,
\]
which completes the proof of (\ref{lem:Poiss_aprox_resultado1}).

Now we turn to estimate (\ref{lem:Poiss_aprox_resultado3-1}). Let $\psi
_{N}=\sum_{i=1}^{N}\lambda_{i}A_{i}$ with $N\in\mathbb{N}$ fixed, and let
$\rho_{\varepsilon}\in C_{c}^{\infty}(\mathbb{R}^{n-1})$ be a mollifier with
$\varepsilon>0$ to be chosen later. Then, we have that
\begin{equation}
|\langle P_{t}\ast v_{0},\psi\rangle|\leq|\langle P_{t}\ast v_{0},\psi
-\psi_{N}\rangle|+|\langle P_{t}\ast v_{0},\psi_{N}-\rho_{\varepsilon}\ast
\psi_{N}\rangle|+|\langle P_{t}\ast v_{0},\rho_{\varepsilon}\ast\psi
_{N}\rangle|. \label{aux-P-est-1}%
\end{equation}
Using duality (\ref{dualidad_M_H}) and estimate (\ref{eq:P_kernel_est2}), we
can estimate
\begin{align*}
|\langle P_{t}\ast v_{0},\psi-\psi_{N}\rangle|  &  \leq\Vert P_{t}\ast
v_{0}\Vert_{\mathcal{M}_{q,\mu}(\mathbb{R}^{n-1})}\Vert\psi-\psi_{N}%
\Vert_{\mathcal{H}_{q^{\prime},\mu}(\mathbb{R}^{n-1})}\\
&  \leq c_{n}t^{-\gamma_{1}}\left\Vert \frac{1}{|x|^{n-1-\gamma_{1}}}\ast
v_{0}\right\Vert _{\mathcal{M}_{q,\mu}(\mathbb{R}^{n-1})}\sum_{i=N+1}^{\infty
}|\lambda_{i}|\\
&  \leq Ct^{-\gamma_{1}}\Vert v_{0}\Vert_{\mathcal{M}_{p,\mu}(\mathbb{R}%
^{n-1})}\sum_{i=N+1}^{\infty}|\lambda_{i}|,
\end{align*}
where in the last inequality we have used (\ref{xiao_liu_n}) with $\gamma
_{1}=\frac{n-1-\mu}{p}-\frac{n-1-\mu}{q}$. For the second term in
(\ref{aux-P-est-1}), we proceed similarly and obtain
\[
|\langle P_{t}\ast v_{0},\psi_{N}-\rho_{\varepsilon}\ast\psi_{N}\rangle|\leq
Ct^{-\gamma_{1}}\Vert v_{0}\Vert_{\mathcal{M}_{p,\mu}(\mathbb{R}^{n-1})}%
\Vert\rho_{\varepsilon}\ast\psi_{N}-\psi_{N}\Vert_{\mathcal{H}_{q^{\prime}%
,\mu}(\mathbb{R}^{n-1})}.
\]

Next, observe that the convergence result in (\ref{lem:Poiss_aprox_resultado1}%
) remains valid if the Poisson kernel $P_{t}$ is replaced by any standard
mollifier $\rho_{\varepsilon}\in C_{c}^{\infty}(\mathbb{R}^{\,n-1})$ with
$\Vert\rho_{\varepsilon}\Vert_{L^{1}}=1$ and $\mathrm{supp}(\rho_{\varepsilon
})\subset B(0,\varepsilon)$. In fact, the proof becomes even simpler in this
case, since $\rho_{\varepsilon}$ vanishes outside $B(0,\varepsilon)$, whereas
the Poisson kernel merely decays and is not compactly supported. With this
observation in hand, we can choose $\varepsilon>0$ sufficiently small
(depending on $N$) so that
\[
\Vert\rho_{\varepsilon}\ast\psi_{N}-\psi_{N}\Vert_{\mathcal{H}_{q^{\prime}%
,\mu}(\mathbb{R}^{n-1})}\leq\sum_{i=N+1}^{\infty}|\lambda_{i}|,
\]
which is possible due to the convergence of mollifiers on finite linear
combinations of blocks. It then follows that
\[
|\langle P_{t}\ast v_{0},\psi_{N}-\rho_{\varepsilon}\ast\psi_{N}\rangle|\leq
Ct^{-\gamma_{1}}\Vert v_{0}\Vert_{\mathcal{M}_{p,\mu}(\mathbb{R}^{n-1})}%
\sum_{i=N+1}^{\infty}|\lambda_{i}|.
\]
For the third term in (\ref{aux-P-est-1}), we use (\ref{lem:Poisson_finitud})
to obtain
\begin{align*}
|\langle P_{t}\ast v_{0},\rho_{\varepsilon}\ast\psi_{N}\rangle|  &  =|\langle
v_{0},P_{t}\ast(\rho_{\varepsilon}\ast\psi_{N})\rangle|\\
&  \leq\Vert v_{0}\Vert_{\mathcal{M}_{p,\mu}(\mathbb{R}^{n-1})}\sum_{i=1}%
^{N}\Vert P_{t}\ast(\rho_{\varepsilon}\ast A_{i})\Vert_{\mathcal{H}%
_{p^{\prime},\mu}(\mathbb{R}^{n-1})}\\
&  \leq\Vert v_{0}\Vert_{\mathcal{M}_{p,\mu}(\mathbb{R}^{n-1})}\sum_{i=1}%
^{N}(C_{1,i}+C_{2,i}t),
\end{align*}
where $C_{1,i},C_{2,i}$ depend on the finite collection $\{R_{i}\}_{i\leq N}$,
and the mollifier.

Hence, combining the estimates for all three terms in (\ref{aux-P-est-1}), we
obtain
\[
|\langle P_{t}\ast v_{0},\psi\rangle|\leq Ct^{-\gamma_{1}}\Vert v_{0}%
\Vert_{\mathcal{M}_{p,\mu}(\mathbb{R}^{n-1})}\sum_{i=N+1}^{\infty}|\lambda
_{i}|+\sum_{i=1}^{N}(C_{1,i}+C_{2,i}t)\Vert v_{0}\Vert_{\mathcal{M}_{p,\mu
}(\mathbb{R}^{n-1})}%
\]
which establishes (\ref{lem:Poiss_aprox_resultado3-1}).

Finally, the derivation of estimate (\ref{lem:Poiss_aprox_resultado2}) follows
by adapting the arguments used for $P_{t}$ in
(\ref{lem:Poiss_aprox_resultado3-1}) and using (\ref{xiao_liu_n_1}) together
with the extension (\ref{extension1}) and $\bar{P}_{x_{n}+t}$ defined in
(\ref{P_extendido}). This completes the proof of the lemma, and we leave the
detailed verification of (\ref{lem:Poiss_aprox_resultado2}) to the reader.\qed

\subsection{Nonlinear estimates}

\hspace{0.3cm} \label{Sec4.2}This section is devoted to contractive-type
estimates for the nonlinear terms of the integral formulation (\ref{int1}).
These terms can be expressed as compositions of the linear operators
$I_{2},I_{3},I_{4}$ with the corresponding power-type nonlinearities.

The first lemma of this section deal with the composition $I_{2}[u_{0}%
|u_{0}|^{p_{2}-1}].$

\begin{lemma}
Let $p_{1},p_{2}\in(1,\infty)$, $\mu\in\lbrack0,n-1)$, and let $q_{0}%
,q_{1},q_{2},\alpha,\beta$ be as in (\ref{Cond-H0}). Suppose further that
$p_{2}<q_{2},$ $p_{1}=2p_{2}-1,$ and $0<\alpha<\beta p_{2}<1$. Then,
considering $C_{2}>0$ as in Lemma \ref{Lem-I2}, we have that
\begin{align}
&  \left\Vert I_{2}[u_{0}|u_{0}|^{p_{2}-1}]-I_{2}[v_{0}|v_{0}|^{p_{2}%
-1}]\right\Vert _{\mathcal{X}}\leq p_{2}C_{2}\left(  \sup_{t>0}t^{\beta
}\left\Vert (u-v)(\cdot,0,t)\right\Vert _{\mathcal{M}_{q_{2},\mu}%
(\mathbb{R}^{n-1})}\right) \nonumber\\
&  \quad\times\left[  \left(  \sup_{t>0}t^{\beta}\left\Vert u(\cdot
,0,t)\right\Vert _{\mathcal{M}_{q_{2},\mu}(\mathbb{R}^{n-1})}\right)
^{p_{2}-1}+\left(  \sup_{t>0}t^{\beta}\left\Vert v(\cdot,0,t)\right\Vert
_{\mathcal{M}_{q_{2},\mu}(\mathbb{R}^{n-1})}\right)  ^{p_{2}-1}\right]  ,
\label{nonlinear_2}%
\end{align}
for all $t^{\beta}u,t^{\beta}v\in L^{\infty}((0,\infty);\mathcal{M}_{q_{2}%
,\mu}(\mathbb{R}^{n-1})).$
\end{lemma}

\textbf{Proof.} We begin by recalling the notation $h_{0}=h(x^{\prime},0,t)$
and defining
\begin{equation}
g(x^{\prime},t)=u(x^{\prime},0,t)|u(x^{\prime},0,t)|^{p_{2}-1}-v(x^{\prime
},0,t)|v(x^{\prime},0,t)|^{p_{2}-1}\text{, for }x^{\prime}\in\mathbb{R}%
^{n-1}\text{ and }t>0. \label{def-g-1}%
\end{equation}
A pointwise estimate yields
\[
|g(x^{\prime},t)|\leq p_{2}\,|(u-v)(x^{\prime},0,t)|\left(  |u(x^{\prime
},0,t)|^{p_{2}-1}+|v(x^{\prime},0,t)|^{p_{2}-1}\right)  .
\]
Applying H\"{o}lder's inequality in Morrey spaces (see (\ref{Holder1})) with
$1<\frac{q_{2}}{p_{2}}<\frac{q_{2}}{p_{2}-1}$ and the relation $\frac{1}%
{q_{2}/p_{2}}=\frac{1}{q_{2}}+\frac{1}{q_{2}/(p_{2}-1)}$, we obtain%
\begin{align}
\Vert g(\cdot,t)\Vert_{\mathcal{M}_{\frac{q_{2}}{p_{2}},\mu}(\mathbb{R}%
^{n-1})}  &  \leq p_{2}\Vert(u-v)(\cdot,0,t)\Vert_{\mathcal{M}_{q_{2},\mu
}(\mathbb{R}^{n-1})}\nonumber\\
&  \times\left(  \Vert|u(\cdot,0,t)|^{p_{2}-1}\Vert_{\mathcal{M}_{\frac{q_{2}%
}{p_{2}-1},\mu}(\mathbb{R}^{n-1})}+\Vert|v(\cdot,0,t)|^{p_{2}-1}%
\Vert_{\mathcal{M}_{\frac{q_{2}}{p_{2}-1},\mu}(\mathbb{R}^{n-1})}\right)
\nonumber\\
&  \leq p_{2}\Vert(u-v)(\cdot,0,t)\Vert_{\mathcal{M}_{q_{2},\mu}%
(\mathbb{R}^{n-1})}\left(  \Vert u(\cdot,0,t)\Vert_{\mathcal{M}_{q_{2},\mu
}(\mathbb{R}^{n-1})}^{p_{2}-1}+\Vert v(\cdot,0,t)\Vert_{\mathcal{M}_{q_{2}%
,\mu}(\mathbb{R}^{n-1})}^{p_{2}-1}\right)  , \label{c_nonlinear_2_1}%
\end{align}
for all $t>0.$ Now, the linear estimate (\ref{lineal_2}) with $g(x^{\prime
},t)$ defined in (\ref{def-g-1}) leads to
\begin{align}
\left\Vert I_{2}[u_{0}|u_{0}|^{p_{2}-1}]-I_{2}[v_{0}|v_{0}|^{p_{2}%
-1}]\right\Vert _{\mathcal{X}}  &  =\left\Vert I_{2}[g]\right\Vert
_{\mathcal{X}}\nonumber\\
&  \leq C_{2}\left(  \sup_{t>0}t^{\beta p_{2}}\left\Vert g(\cdot,t)\right\Vert
_{\mathcal{M}_{\frac{q_{2}}{p_{2}},\mu}(\mathbb{R}^{n-1})}\right)  .
\label{c_nonlinear_2_2}%
\end{align}
Noting that $t^{\beta p_{2}}=t^{\beta}\cdot(t^{\beta})^{p_{2}-1},$ we conclude
(\ref{nonlinear_2}) from estimates (\ref{c_nonlinear_2_1}) and
(\ref{c_nonlinear_2_2}).

\qed

In the sequel, we derive an estimate for the nonlinear term $I_{3}%
[u|u|^{p_{1}-1}].$

\begin{lemma}
Let $p_{1},p_{2}\in(1,\infty)$, $\mu\in\lbrack0,n-1)$, and let $q_{0}%
,q_{1},q_{2},\alpha,\beta$ be as in (\ref{Cond-H0}). Suppose further that
$p_{1}<q_{1},$ $p_{1}=2p_{2}-1,$ and $0<\beta<\alpha p_{1}<1$. Then, we have
the estimate
\begin{align}
&  \left\Vert I_{3}[u|u|^{p_{1}-1}]-I_{3}[v|v|^{p_{1}-1}]\right\Vert
_{\mathcal{X}}\leq p_{1}C_{3}\left(  \sup_{t>0}t^{\alpha}\left\Vert
(u-v)(\cdot,t)\right\Vert _{\mathcal{M}_{q_{1},\mu}(\mathbb{R}_{+}^{n}%
)}\right) \nonumber\\
&  \qquad\qquad\times\left[  \left(  \sup_{t>0}t^{\alpha}\left\Vert
u(\cdot,t)\right\Vert _{\mathcal{M}_{q_{1},\mu}(\mathbb{R}_{+}^{n})}\right)
^{p_{1}-1}+\left(  \sup_{t>0}t^{\alpha}\left\Vert v(\cdot,t)\right\Vert
_{\mathcal{M}_{q_{1},\mu}(\mathbb{R}_{+}^{n})}\right)  ^{p_{1}-1}\right]  ,
\label{nonlinear_3}%
\end{align}
for all $t^{\alpha}u,t^{\alpha}v\in L^{\infty}((0,\infty);\mathcal{M}%
_{q_{1},\mu}(\mathbb{R}_{+}^{n})),$ where $C_{3}>0$ is as in Lemma
\ref{Lem-I3}.
\end{lemma}

\textbf{Proof.} Considering
\begin{equation}
f(x,t)=u(x,t)|u(x,t)|^{p_{1}-1}-v(x,t)|v(x,t)|^{p_{1}-1}, \label{def-f-1}%
\end{equation}
we have the pointwise estimate
\[
|f(x,t)|\leq p_{1}\,|u(x,t)-v(x,t)|\left(  |u(x,t)|^{p_{1}-1}+|v(x,t)|^{p_{1}%
-1}\right)  ,\text{ for }x\in\mathbb{R}_{+}^{n}\text{ and }t>0.
\]
By H\"{o}lder's inequality (\ref{Holder1}) with $1<\frac{q_{1}}{p_{1}}%
<\frac{q_{1}}{p_{1}-1}$ and $\frac{1}{q_{1}/p_{1}}=\frac{1}{q_{1}}+\frac
{1}{q_{1}/(p_{1}-1)}$, we can estimate the norm of $f$ as follows:%
\begin{align}
&  \Vert f(\cdot,t)\Vert_{\mathcal{M}_{\frac{q_{1}}{p_{1}},\mu}(\mathbb{R}%
_{+}^{n})}\nonumber\\
&  \leq p_{1}\Vert(u-v)(\cdot,t)\Vert_{\mathcal{M}_{q_{1},\mu}(\mathbb{R}%
_{+}^{n})}\left(  \Vert|u(\cdot,t)|^{p_{1}-1}\Vert_{\mathcal{M}_{\frac{q_{1}%
}{p_{1}-1},\mu}(\mathbb{R}_{+}^{n})}+\Vert|v(\cdot,t)|^{p_{1}-1}%
\Vert_{\mathcal{M}_{\frac{q_{1}}{p_{1}-1},\mu}(\mathbb{R}_{+}^{n})}\right)
\nonumber\\
&  \leq p_{1}\Vert(u-v)(\cdot,t)\Vert_{\mathcal{M}_{q_{1},\mu}(\mathbb{R}%
_{+}^{n})}\left(  \Vert u(\cdot,t)\Vert_{\mathcal{M}_{q_{1},\mu}%
(\mathbb{R}_{+}^{n})}^{p_{1}-1}+\Vert v(\cdot,t)\Vert_{\mathcal{M}_{q_{1},\mu
}(\mathbb{R}_{+}^{n})}^{p_{1}-1}\right)  . \label{c_nonlinear_3_1}%
\end{align}
Employing the linear estimate (\ref{lineal_3}) with $f(x,t)$ as in
(\ref{def-f-1}), we obtain
\begin{align}
\left\Vert I_{3}[u|u|^{p_{1}-1}-v|v|^{p_{1}-1}]\right\Vert _{\mathcal{X}}  &
=\left\Vert I_{3}[f]\right\Vert _{\mathcal{X}}\nonumber\\
&  \leq C_{3}\left(  \sup_{t>0}t^{\alpha p_{1}}\left\Vert f(\cdot
,t)\right\Vert _{\mathcal{M}_{\frac{q_{1}}{p_{1}},\mu}(\mathbb{R}_{+}^{n}%
)}\right)  . \label{c_nonlinear_3_2}%
\end{align}
Inserting (\ref{c_nonlinear_3_1}) into (\ref{c_nonlinear_3_2}) and observing
that $t^{\alpha p_{1}}=t^{\alpha}\cdot(t^{\alpha})^{p_{1}-1}$ yields
(\ref{nonlinear_3}).

\qed

Next, we turn to establishing an estimate for the nonlinear term
$I_{4}[u|u|^{p_{1}-1}]$.

\begin{lemma}
Let $p_{1},p_{2}\in(1,\infty)$, $\mu\in\lbrack0,n-2)$, and let $q_{0}%
,q_{1},q_{2},\alpha,\beta$ be as in (\ref{Cond-H0}). In addition, assume that
$\frac{n-\mu}{n-\mu-2}<p_{1}<q_{1}$, and $\alpha,\beta>0$. Then, considering
$C_{4}>0$ as in Lemma \ref{Lem-I4}, we have that
\begin{align}
&  \left\Vert I_{4}[u|u|^{p_{1}-1}]-I_{4}[v|v|^{p_{1}-1}]\right\Vert
_{\mathcal{X}}\nonumber\\
&  \leq p_{1}C_{4}\left(  \sup_{t>0}t^{\alpha}\left\Vert (u-v)(\cdot
,t)\right\Vert _{\mathcal{M}_{q_{1},\mu}(\mathbb{R}_{+}^{n})}+\sup
_{t>0}\left\Vert (u-v)(\cdot,t)\right\Vert _{\mathcal{M}_{q_{0},\mu
}(\mathbb{R}_{+}^{n})}\right) \nonumber\\
&  \times\left[  \left(  \sup_{t>0}\left\Vert u(\cdot,t)\right\Vert
_{\mathcal{M}_{q_{0},\mu}(\mathbb{R}_{+}^{n})}\right)  ^{p_{1}-1}+\left(
\sup_{t>0}\left\Vert v(\cdot,t)\right\Vert _{\mathcal{M}_{q_{0},\mu
}(\mathbb{R}_{+}^{n})}\right)  ^{p_{1}-1}\right]  \label{nonlinear_4}%
\end{align}
for all $u,v\in L^{\infty}((0,\infty);\mathcal{M}_{q_{0},\mu}(\mathbb{R}%
_{+}^{n}))$ with $t^{\alpha}u,t^{\alpha}v\in L^{\infty}((0,\infty
);\mathcal{M}_{q_{1},\mu}(\mathbb{R}_{+}^{n})).$
\end{lemma}

\textbf{Proof.} Again we consider $f$ as in (\ref{def-f-1}) and recall the
pointwise estimate%
\[
|f(x,t)|\leq p_{1}\,|u(x,t)-v(x,t)|\left(  |u(x,t)|^{p_{1}-1}+|v(x,t)|^{p_{1}%
-1}\right)  .
\]
In view of the estimate (\ref{lineal_4}) for $I_{4}$, we require bounds on
$f_{1}$ in two different Morrey norms. Using H\"{o}lder's inequality
(\ref{Holder1}) with $\frac{n-\mu}{2}>1,$ $\frac{(n-\mu)q_{1}}{n-\mu+2q_{1}%
}>1$ (since $\frac{n-\mu}{n-\mu-2}<q_{1}$), and the relation
\[
\frac{1}{\frac{(n-\mu)q_{1}}{n-\mu+2q_{1}}}=\frac{1}{q_{1}}+\frac{1}%
{\frac{n-\mu}{2}},
\]
we obtain%
\begin{align}
&  \Vert f(\cdot,t)\Vert_{\mathcal{M}_{\frac{(n-\mu)q_{1}}{n-\mu+2q_{1}},\mu
}(\mathbb{R}_{+}^{n})}\nonumber\\
&  \leq p_{1}\Vert u-v\Vert_{\mathcal{M}_{q_{1},\mu}(\mathbb{R}_{+}^{n}%
)}\left(  \Vert|u|^{p_{1}-1}\Vert_{\mathcal{M}_{\frac{n-\mu}{2},\mu
}(\mathbb{R}_{+}^{n})}+\Vert|v|^{p_{1}-1}\Vert_{\mathcal{M}_{\frac{n-\mu}%
{2},\mu}(\mathbb{R}_{+}^{n})}\right) \nonumber\\
&  \leq p_{1}\Vert u-v\Vert_{\mathcal{M}_{q_{1},\mu}(\mathbb{R}_{+}^{n}%
)}\left(  \Vert u\Vert_{\mathcal{M}_{q_{0},\mu}(\mathbb{R}_{+}^{n})}^{p_{1}%
-1}+\Vert v\Vert_{\mathcal{M}_{q_{0} ,\mu}(\mathbb{R}_{+}^{n})}^{p_{1}%
-1}\right)  . \label{c_nonlinear_4_1}%
\end{align}
Moreover, by (\ref{Holder1}) with $\frac{n-\mu}{2}>1$, $\frac{q_{0}}{p_{1}}>1$
(since $\frac{n-\mu}{n-\mu-2}<p_{1}$), and the relation $\frac{1}{q_{0}/p_{1}%
}=\frac{1}{q_{0}}+\frac{1}{(n-\mu)/2}$, it follows that
\begin{align}
&  \Vert f(\cdot,t)\Vert_{\mathcal{M}_{\frac{q_{0}}{p_{1}},\mu}(\mathbb{R}%
_{+}^{n})}\nonumber\\
&  \leq p_{1}\Vert u-v\Vert_{\mathcal{M}_{q_{0},\mu}(\mathbb{R}_{+}^{n}%
)}\left(  \Vert|u|^{p_{1}-1}\Vert_{\mathcal{M}_{\frac{n-\mu}{2},\mu
}(\mathbb{R}_{+}^{n})}+\Vert|v|^{p_{1}-1}\Vert_{\mathcal{M}_{\frac{n-\mu}%
{2},\mu}(\mathbb{R}_{+}^{n})}\right) \nonumber\\
&  \leq p_{1}\Vert u-v\Vert_{\mathcal{M}_{q_{0},\mu}(\mathbb{R}_{+}^{n}%
)}\left(  \Vert u\Vert_{\mathcal{M}_{q_{0} ,\mu}(\mathbb{R}_{+}^{n})}%
^{p_{1}-1}+\Vert v\Vert_{\mathcal{M}_{q_{0},\mu}(\mathbb{R}_{+}^{n})}%
^{p_{1}-1}\right)  . \label{c_nonlinear_4_2}%
\end{align}

Now, estimate (\ref{lineal_4}) with $f(x,t)$ as in (\ref{def-f-1}) yields
\begin{align}
&  \left\Vert I_{4}[u|u|^{p_{1}-1}]-I_{4}[v|v|^{p_{1}-1}]\right\Vert
_{\mathcal{X}}=\left\Vert I_{4}[f]\right\Vert _{\mathcal{X}}\nonumber\\
&  \leq C_{4}\left(  \sup_{t>0}t^{\alpha}\left\Vert f(\cdot,t)\right\Vert
_{\mathcal{M}_{\frac{(n-\mu)q_{1}}{n-\mu+2q_{1}},\mu}(\mathbb{R}_{+}^{n}%
)}+\sup_{t>0}\left\Vert f(\cdot,t)\right\Vert _{\mathcal{M}_{\frac{q_{0}%
}{p_{1}},\mu}(\mathbb{R}_{+}^{n})}\right)  \label{c_nonlinear_4_3}%
\end{align}
Considering (\ref{c_nonlinear_4_1}) and (\ref{c_nonlinear_4_2}) in conjunction
with (\ref{c_nonlinear_4_3}), we deduce (\ref{nonlinear_4}). \qed

\section{\bigskip Proofs of main results}

\label{Sec5}With the estimates developed in Sections \ref{Sec4.1} and
\ref{Sec4.2} in hand, we are now in a position to carry out the proof of the
theorems stated in Section \ref{Sec3}.

\subsection{Proof of Theorem \ref{theo}}

\label{Sec5.1}Let $\varepsilon>0$ and $\delta=\frac{\varepsilon}{C_{1}},$
where the constant $C_{1}$ is as in Lemma \ref{Lem-I1}, and recall that
$\widetilde{q}_{0}=(n-1-\mu)(p_{1}-1)/2.$ By hypotheses, we have that
\begin{equation}
\left\Vert \varphi\right\Vert _{\mathcal{M}_{\widetilde{q}_{0},\mu}%
(\mathbb{R}^{n-1})}\leq\delta. \label{valinicial_delta}%
\end{equation}
Also, consider the closed ball $D_{2\varepsilon}=\{u\in\mathcal{\mathcal{X}%
};\Vert u\Vert_{\mathcal{X}}\leq2\varepsilon\}$ equipped with the complete
metric $\mathcal{Z}(u,v)=\Vert u-v\Vert_{\mathcal{X}}$, where the norm
$\Vert\cdot\Vert_{\mathcal{X}}$ is defined in (\ref{norm1}). For some suitable
$\varepsilon>0$, we are going to show that the mapping
\begin{equation}
\Phi(u):=I_{1}[\varphi]+I_{2}[u_{0}|u_{0}|^{p_{2}-1}]+I_{3}[u|u|^{p_{1}%
-1}]+I_{4}[u|u|^{p_{1}-1}] \label{map-1}%
\end{equation}
is a contraction on $(D_{2\varepsilon},\mathcal{Z})$.

From condition (\ref{valinicial_delta}), the linear estimate (\ref{lineal_1}),
and estimates (\ref{nonlinear_2}), (\ref{nonlinear_3}) and (\ref{nonlinear_4})
with $v=0$, it follows that
\begin{align}
\left\Vert \Phi\lbrack u]\right\Vert _{\mathcal{\mathcal{X}}}  &  \leq
C_{1}\left\Vert \varphi\right\Vert _{\mathcal{M}_{\widetilde{q}_{0},\mu
}(\mathbb{R}^{n-1})}+p_{2}C_{2}\left(  \left\Vert u\right\Vert _{\mathcal{X}%
}\right)  ^{p_{2}}+p_{1}(C_{3}+C_{4})\left(  \left\Vert u\right\Vert
_{\mathcal{X}}\right)  ^{p_{1}}\nonumber\\
&  \leq C_{1}\delta+p_{2}C_{2}2^{p_{2}}\varepsilon^{p_{2}}+p_{1}(C_{3}%
+C_{4})2^{p_{1}}\varepsilon^{p_{1}}\nonumber\\
&  =2\varepsilon\left(  \frac{1}{2}+p_{2}C_{2}2^{p_{2}-1}\varepsilon^{p_{2}%
-1}+p_{1}(C_{3}+C_{4})2^{p_{1}-1}\varepsilon^{p_{1}-1}\right)  ,
\label{des_phi_linear}%
\end{align}
for all $u\in D_{2\varepsilon}$. Moreover, using (\ref{nonlinear_2}),
(\ref{nonlinear_3}) and (\ref{nonlinear_4}), we obtain that%

\begin{align}
&  \Vert\Phi(u)-\Phi(v)\Vert_{\mathcal{X}}\nonumber\\
&  \leq\left\Vert I_{2}[u_{0}|u_{0}|^{p_{2}-1}-v_{0}|v_{0}|^{p_{2}%
-1}]\right\Vert _{\mathcal{X}}+\left\Vert I_{3}[u|u|^{p_{1}-1}-v|v|^{p_{1}%
-1}]\right\Vert _{\mathcal{X}}+\left\Vert I_{4}[u|u|^{p_{1}-1}-v|v|^{p_{1}%
-1}]\right\Vert _{\mathcal{X}}\nonumber\\
&  \leq p_{2}C_{2}\left(  \Vert u-v\Vert_{\mathcal{X}}\right)  \left(  \Vert
u\Vert_{\mathcal{X}}^{p_{2}-1}+\Vert v\Vert_{\mathcal{X}}^{p_{2}-1}\right)
+p_{1}(C_{3}+C_{4})\left(  \Vert u-v\Vert_{\mathcal{X}}\right)  \left(  \Vert
u\Vert_{\mathcal{X}}^{p_{1}-1}+\Vert v\Vert_{\mathcal{X}}^{p_{1}-1}\right)
\nonumber\\
&  \leq2\left(  p_{2}C_{2}2^{p_{2}-1}\varepsilon^{p_{2}-1}+p_{1}(C_{3}%
+C_{4})2^{p_{1}-1}\varepsilon^{p_{1}-1}\right)  \left\Vert u-v\right\Vert
_{\mathcal{X}}\text{,} \label{des_phi_bilinear}%
\end{align}
for all $u,v\in D_{2\varepsilon}$.

Choosing now $\varepsilon>0$ such that
\begin{equation}
0<\eta=2\left(  p_{2}C_{2}2^{p_{2}-1}\varepsilon^{p_{2}-1}+p_{1}(C_{3}%
+C_{4})2^{p_{1}-1}\varepsilon^{p_{1}-1}\right)  <1, \label{q_menorq1}%
\end{equation}
and using (\ref{des_phi_linear}), we arrive at
\begin{align}
\left\Vert \Phi\lbrack u]\right\Vert _{\mathcal{X}}  &  \leq2\varepsilon
\left(  \frac{1}{2}+p_{2}C_{2}2^{p_{2}-1}\varepsilon^{p_{2}-1}+p_{1}%
(C_{3}+C_{4})2^{p_{1}-1}\varepsilon^{p_{1}-1}\right) \nonumber\\
&  <2\varepsilon\left(  \frac{1}{2}+\frac{1}{2}\right)  =2\varepsilon,\text{
for all }u\in D_{2\varepsilon}. \label{aux-proof1}%
\end{align}
Also, we obtain from (\ref{des_phi_bilinear}) and (\ref{q_menorq1}) that%
\begin{equation}
\left\Vert \Phi\lbrack u]-\Phi\lbrack v]\right\Vert _{\mathcal{X}}\leq
\eta\left\Vert u-v\right\Vert _{\mathcal{X}},\text{ for all }u\in
D_{2\varepsilon}\text{.} \label{aux-proof2}%
\end{equation}
Since $0<\eta<1$, estimates (\ref{aux-proof1}) and (\ref{aux-proof2}) imply
that $\Phi$ maps $D_{2\varepsilon}$ into itself and acts as a contraction on
this set. Consequently, by the Contraction Mapping Principle, there exists a
unique fixed point $u\in D_{2\varepsilon}$ of $\Phi$. This fixed point is
therefore the unique solution to the integral equation (\ref{int1}) satisfying
$\Vert u\Vert_{\mathcal{X}}\leq2\varepsilon$.

In next step, we address the continuous dependence on the initial data
$\varphi$. For that, let $u$ and $\tilde{u}\in$ $D_{2\varepsilon}$ be the
solutions corresponding to initial data $\varphi$ and $\tilde{\varphi}%
\in\mathcal{M}_{\widetilde{q}_{0},\mu}(\mathbb{R}^{n-1})$, respectively. Using
the linear estimate (\ref{lineal_1}) and the contractive estimate
(\ref{des_phi_bilinear}), we can estimate
\begin{align*}
\Vert u-\tilde{u}\Vert_{\mathcal{X}}  &  \leq\left\Vert I_{1}[\varphi
-\tilde{\varphi}]\right\Vert _{\mathcal{X}}+\left\Vert I_{2}[u_{0}%
|u_{0}|^{p_{2}-1}-\tilde{u}_{0}|\tilde{u}_{0}|^{p_{2}-1}]\right\Vert
_{\mathcal{X}}+\left\Vert I_{3}[u|u|^{p_{1}-1}-\tilde{u}|\tilde{u}|^{p_{1}%
-1}]\right\Vert _{\mathcal{X}}\\
&  +\left\Vert I_{4}[u|u|^{p_{1}-1}-\tilde{u}|\tilde{u}|^{p_{1}-1}]\right\Vert
_{\mathcal{X}}\\
&  \leq C_{1}\left\Vert \varphi-\tilde{\varphi}\right\Vert _{,\mathcal{M}%
_{\widetilde{q}_{0},\mu}(\mathbb{R}^{n-1})}+\eta\left\Vert u-\tilde
{u}\right\Vert _{\mathcal{X}},
\end{align*}
where $0<\eta<1$. Rearranging terms, it follows that
\[
\Vert u-\tilde{u}\Vert_{\mathcal{X}}\leq\frac{C_{1}}{1-\eta}\left\Vert
\varphi-\tilde{\varphi}\right\Vert _{\mathcal{M}_{\widetilde{q}_{0},\mu
}(\mathbb{R}^{n-1})}%
\]
which yields the desired Lipschitz continuity of the data-to-solution map.

Finally, we prove the convergence of the solution trace $u(x^{\prime}%
,x_{n},t)|_{x_{n}=0}$ to the initial data $\varphi$, as $t\rightarrow0^{+},$
in the weak-$\ast$ topology of $\mathcal{M}_{\widetilde{q}_{0},\mu}%
(\mathbb{R}^{n-1})$. Evaluating equation (\ref{int1}) at $x_{n}=0$, we have
that
\[
u_{0}=I_{1}^{0}[\varphi]+I_{2}^{0}[u_{0}|u_{0}|^{p_{2}-1}]+I_{3}%
^{0}[u|u|^{p_{1}-1}]+I_{4}^{0}[u|u|^{p_{1}-1}],
\]
where $I_{j}^{0}[\cdot]=I_{j}[\cdot]\big|_{x_{n}=0}$ for $j=1,2,3,4$. First
observe that $I_{4}^{0}[u|u|^{p_{1}-1}]=0$ since $G(x^{\prime},0,y)=0$.
Throughout this part of the proof, we use the notation $q=\widetilde{q}%
_{0}=\frac{(n-1-\mu)(p_{1}-1)}{2}$.

Let $\psi\in\mathcal{H}_{q^{\prime},\mu}(\mathbb{R}^{n-1})$ with block
decomposition $\psi=\sum_{i=1}^{\infty}\lambda_{i}A_{i}$ as in
(\ref{def_blocks_space}), where $\sum_{i=1}^{\infty}|\lambda_{i}|<\infty$.
Then,
\begin{equation}
\left\langle u_{0},\psi\right\rangle =\left\langle I_{1}^{0}[\varphi
],\psi\right\rangle +\left\langle I_{2}^{0}[u_{0}|u_{0}|^{p_{2}-1}%
],\psi\right\rangle +\left\langle I_{3}^{0}[u|u|^{p_{1}-1}],\psi\right\rangle
. \label{lim_sum_0}%
\end{equation}
We begin with the first term $I_{1}^{0}[\varphi]=I_{1}[\varphi](x^{\prime
},0,t)$. By Young's inequality, we have that%
\begin{equation}
\left\Vert I_{1}^{0}[\varphi]\right\Vert _{\mathcal{M}_{q,\mu}(\mathbb{R}%
^{n-1})}=\left\Vert P(\cdot,t)\ast\varphi(\cdot)\right\Vert _{\mathcal{M}%
_{q,\mu}(\mathbb{R}^{n-1})}\leq\left\Vert \varphi\right\Vert _{\mathcal{M}%
_{q,\mu}(\mathbb{R}^{n-1})}. \label{I_1^0_bound}%
\end{equation}
Defining $\psi_{N}=\sum_{i=1}^{N}\lambda_{i}A_{i}$, and using duality
(\ref{dualidad_M_H}) and (\ref{I_1^0_bound}), we obtain that \
\begin{align}
|\left\langle I_{1}^{0}[\varphi],\psi\right\rangle -\left\langle \varphi
,\psi\right\rangle |  &  \leq|\left\langle I_{1}^{0}[\varphi],\psi-\psi
_{N}\right\rangle |+|\left\langle \varphi,\psi-\psi_{N}\right\rangle
|+|\left\langle I_{1}^{0}[\varphi],\psi_{N}\right\rangle -\left\langle
\varphi,\psi_{N}\right\rangle |\nonumber\\
&  \leq2\Vert\varphi\Vert_{\mathcal{M}_{q,\mu}(\mathbb{R}^{n-1})}\sum
_{i=N+1}^{\infty}|\lambda_{i}|+|\left\langle I_{1}^{0}[\varphi],\psi
_{N}\right\rangle -\left\langle \varphi,\psi_{N}\right\rangle |. \label{lim_2}%
\end{align}
Since $\sum_{i=N+1}^{\infty}\lambda_{i}\rightarrow0$, as $N\rightarrow\infty$,
we can choose $N$ so large that
\begin{equation}
2\Vert\varphi\Vert_{\mathcal{M}_{q,\mu}(\mathbb{R}^{n-1})}\sum_{i=N+1}%
^{\infty}|\lambda_{i}|<\frac{\varepsilon}{2}. \label{lim_3}%
\end{equation}
Moreover, for each $(q^{\prime},\mu)$-block $A_{i}$, by
(\ref{lem:Poiss_aprox_resultado1}) we have that
\[
\lim_{t\rightarrow0^{+}}\Vert P_{t}\ast A_{i}-A_{i}\Vert_{\mathcal{H}%
_{q^{\prime},\mu}(\mathbb{R}^{n-1})}=0\text{, for all }i=1,...,N.
\]
Hence, we can choose $\delta(N)>0$ such that
\[
\Vert P_{t}\ast\psi_{N}-\psi_{N}\Vert_{\mathcal{H}_{q^{\prime},\mu}%
(\mathbb{R}^{n-1})}<\frac{\varepsilon}{2\Vert\varphi\Vert_{\mathcal{M}_{q,\mu
}(\mathbb{R}^{n-1})}},\text{ for all }0<t<\delta,
\]
and then
\begin{equation}
|\left\langle I_{1}^{0}[\varphi],\psi_{N}\right\rangle -\left\langle
\varphi,\psi_{N}\right\rangle |\leq\Vert\varphi\Vert_{\mathcal{M}_{q,\mu
}(\mathbb{R}^{n-1})}\Vert P_{t}\ast\psi_{N}-\psi_{N}\Vert_{\mathcal{H}%
_{q^{\prime},\mu}(\mathbb{R}^{n-1})}<\frac{\varepsilon}{2}. \label{lim_4}%
\end{equation}
Since $\varepsilon>0$ is arbitrary, combining (\ref{lim_3}) and (\ref{lim_4})
in (\ref{lim_2}) gives%
\begin{equation}
\lim_{t\rightarrow0^{+}}|\left\langle I_{1}^{0}[\varphi],\psi\right\rangle
-\left\langle \varphi,\psi\right\rangle |=0. \label{lim_1.0}%
\end{equation}

For the second term $I_{2}^{0}[u_{0}|u_{0}|^{p_{2}-1}]=I_{2}[u_{0}%
|u_{0}|^{p_{2}-1}](x^{\prime},0,t)$ in (\ref{lim_sum_0}), we apply
(\ref{lem:Poiss_aprox_resultado3-1}) with $p=\frac{q_{2}}{p_{2}}$, $1-\beta
p_{2}=\frac{n-\mu-1}{\frac{q_{2}}{p_{2}}}-\frac{n-\mu-1}{q}$ and $v_{0}%
=u_{0}|u_{0}|^{p_{2}-1}$ to obtain
\begin{align*}
|\left\langle I_{2}^{0}[u_{0}|u_{0}|^{p_{2}-1}],\psi\right\rangle |  &
\leq\int_{0}^{t}|\left\langle P_{t-s}\ast((u_{0}|u_{0}|^{p_{2}-1}%
)(s)),\psi\right\rangle |ds\\
&  \leq C_{0}\sum_{i=N+1}^{\infty}|\lambda_{i}|\int_{0}^{t}s^{-\beta p_{2}%
}(t-s)^{-1+\beta p_{2}}\left[  s^{\beta p_{2}}\Vert(u_{0}|u_{0}|^{p_{2}%
-1})(s)\Vert_{\mathcal{M}_{\frac{q_{2}}{p_{2}},\mu}(\mathbb{R}^{n-1})}\right]
ds\\
&  \qquad+\int_{0}^{t}s^{-\beta p_{2}}(C_{N}+D_{N}(t-s))\left[  s^{\beta
p_{2}}\Vert(u_{0}|u_{0}|^{p_{2}-1})(s)\Vert_{\mathcal{M}_{\frac{q_{2}}{p_{2}%
},\mu}(\mathbb{R}^{n-1})}\right]  ds.
\end{align*}
The change of variables $s\rightarrow tr$ yields
\begin{align}
|\left\langle I_{2}^{0}[u_{0}|u_{0}|^{p_{2}-1}],\psi\right\rangle |  &  \leq
C_{0}\sum_{i=N+1}^{\infty}|\lambda_{i}|\int_{0}^{1}r^{-\beta p_{2}%
}(1-r)^{-1+\beta p_{2}}dr\left[  \sup_{t>0}t^{\beta}\left\Vert u(\cdot
,0,t)\right\Vert _{\mathcal{M}_{q_{2},\mu}(\mathbb{R}^{n-1})}\right]  ^{p_{2}%
}\nonumber\\
&  \qquad+t^{1-\beta p_{2}}\int_{0}^{1}r^{-\beta p_{2}}(C_{N}+D_{N}%
t(1-r))dr\left[  \sup_{t>0}t^{\beta}\left\Vert u(\cdot,0,t)\right\Vert
_{\mathcal{M}_{q_{2},\mu}(\mathbb{R}^{n-1})}\right]  ^{p_{2}}. \label{lim_2.0}%
\end{align}
We now choose $N$ sufficiently large so that
\begin{equation}
C_{0}\sum_{i=N+1}^{\infty}|\lambda_{i}|\int_{0}^{1}r^{-\beta p_{2}%
}(1-r)^{-1+\beta p_{2}}dr\left[  \sup_{t>0}t^{\beta}\left\Vert u(\cdot
,0,t)\right\Vert _{\mathcal{M}_{q_{2},\mu}(\mathbb{R}^{n-1})}\right]  ^{p_{2}%
}<\frac{\varepsilon}{2}. \label{lim_2.1}%
\end{equation}
Since $t^{1-\beta p_{2}}\rightarrow0^{+}$ as $t\rightarrow0^{+}$ (because
$\beta p_{2}<1$), there exists $\delta=\delta(N)>0$ such that, for all
$0<t<\delta,$
\begin{equation}
t^{1-\beta p_{2}}\int_{0}^{1}r^{-\beta p_{2}}(C_{N}+D_{N}t(1-r))dr\left[
\sup_{t>0}t^{\beta}\left\Vert u(\cdot,0,t)\right\Vert _{\mathcal{M}_{q_{2}%
,\mu}(\mathbb{R}^{n-1})}\right]  ^{p_{2}}<\frac{\varepsilon}{2}.
\label{lim_2.2}%
\end{equation}
Considering (\ref{lim_2.1}) and (\ref{lim_2.2}) in (\ref{lim_2.0}), we obtain
\begin{equation}
\lim_{t\rightarrow0^{+}}|\left\langle I_{2}^{0}[u_{0}|u_{0}|^{p_{2}%
-1}](t),\psi\right\rangle |=0. \label{lim_2.final}%
\end{equation}

In the sequel we deal with the term $I_{3}^{0}[u|u|^{p_{1}-1}]=I_{3}%
[u|u|^{p_{1}-1}](x^{\prime},0,t)$ in (\ref{lim_sum_0}). For that, using
estimate (\ref{lem:Poiss_aprox_resultado2}) with $p=\frac{q_{1}}{p_{1}}$,
$2-\alpha p_{1}=\frac{n-\mu}{\frac{q_{1}}{p_{1}}}-\frac{n-\mu-1}{q}$ and
$v=u|u|^{p_{1}}$, we have that
\begin{align*}
|\left\langle I_{3}^{0}[u|u|^{p_{1}-1}],\psi\right\rangle |  &  \leq\int
_{0}^{t}\left\langle \int_{0}^{\infty}P_{y_{n}+t-s}\ast((u|u|^{p_{1}-1}%
)(\cdot,y_{n},s))dy_{n},\psi\right\rangle ds\\
&  \leq\tilde{C}_{0}\sum_{i=N+1}^{\infty}|\lambda_{i}|\int_{0}^{t}s^{-\alpha
p_{1}}(t-s)^{1-(2-\alpha p_{1})}\left[  s^{\alpha p_{1}}\Vert(u|u|^{p_{1}%
-1})(s)\Vert_{\mathcal{M}_{\frac{q_{1}}{p_{1}},\mu}(\mathbb{R}_{+}^{n}%
)}\right]  ds\\
&  \qquad+\int_{0}^{t}s^{-\alpha p_{1}}(\tilde{C}_{N}+\tilde{D}_{N}%
(t-s))\left[  s^{\alpha p_{1}}\Vert(u|u|^{p_{1}-1})(s)\Vert_{\mathcal{M}%
_{\frac{q_{1}}{p_{1}},\mu}(\mathbb{R}_{+}^{n})}\right]  ds.
\end{align*}
By the change $s\rightarrow tr,$ we get%
\begin{align}
|\left\langle I_{3}^{0}[u|u|^{p_{1}-1}],\psi\right\rangle |  &  \leq\tilde
{C}_{0}\sum_{i=N+1}^{\infty}|\lambda_{i}|\int_{0}^{1}r^{-\alpha p_{1}%
}(1-r)^{-1+\alpha p_{1}}dr\left[  \sup_{t>0}t^{\alpha}\left\Vert
u(\cdot,t)\right\Vert _{\mathcal{M}_{q_{1},\mu}(\mathbb{R}_{+}^{n})}\right]
^{p_{1}}\nonumber\\
&  \qquad+t^{1-\alpha p_{1}}\int_{0}^{1}r^{-\alpha p_{1}}(\tilde{C}_{N}%
+\tilde{D}_{N}t(1-r))dr\left[  \sup_{t>0}t^{\alpha}\left\Vert u(\cdot
,t)\right\Vert _{\mathcal{M}_{q_{1},\mu}(\mathbb{R}_{+}^{n})}\right]  ^{p_{1}%
}. \label{lim_3.0}%
\end{align}
Consider now $N$ sufficiently large so that
\begin{equation}
\tilde{C}_{0}\sum_{i=N+1}^{\infty}|\lambda_{i}|\int_{0}^{1}r^{-\alpha p_{1}%
}(1-r)^{-1+\alpha p_{1}}dr\left[  \sup_{t>0}t^{\alpha}\left\Vert
u(\cdot,t)\right\Vert _{\mathcal{M}_{q_{1},\mu}(\mathbb{R}_{+}^{n})}\right]
^{p_{1}}<\frac{\varepsilon}{2}. \label{lim_3.1}%
\end{equation}
Since $t^{1-\alpha p_{1}}\rightarrow0^{+}$ as $t\rightarrow0^{+}$ (because
$\alpha p_{1}<1$)$,$ we can choose $\delta=\delta(N)>0$ such that, for all
$0<t<\delta,$
\begin{equation}
t^{1-\alpha p_{1}}\int_{0}^{1}r^{-\alpha p_{1}}(\tilde{C}_{N}+\tilde{D}%
_{N}t(1-r))dr\left[  \sup_{t>0}t^{\alpha}\left\Vert u(\cdot,t)\right\Vert
_{\mathcal{M}_{q_{1},\mu}(\mathbb{R}_{+}^{n})}\right]  ^{p_{1}}<\frac
{\varepsilon}{2}. \label{lim_3.2}%
\end{equation}
Combining (\ref{lim_3.1}) and (\ref{lim_3.2}) in (\ref{lim_3.0}), we arrive
at
\begin{equation}
\lim_{t\rightarrow0^{+}}|\left\langle I_{3}^{0}[u|u|^{p_{1}-1}](t),\psi
\right\rangle |=0. \label{lim_3.final}%
\end{equation}
Then, letting $t\rightarrow0^{+}$ in (\ref{lim_sum_0}), and using
(\ref{lim_1.0}), (\ref{lim_2.final}) and (\ref{lim_3.final}), we conclude
that
\[
\lim_{t\rightarrow0^{+}}\left\langle u(\cdot,0,t),\psi\right\rangle
=\langle\varphi,\psi\rangle,\text{ for every }\psi\in\mathcal{H}_{q^{\prime
},\mu}(\mathbb{R}^{n-1}),
\]
which yields the desired weak-$\ast$ convergence. \qed

\subsection{Proof of Theorem \ref{cor:consequences}}

\label{Sec5.2}\textbf{Part (i) - Self-similarity. } By the proof of the
well-posedness result (see Subsection \ref{Sec5.1}), there exists a unique
$u\in D_{2\varepsilon}$ satisfying the integral equation (\ref{int1}), that
is,%
\[
u=I_{1}[\varphi]+I_{2}[u_{0}|u_{0}|^{p_{2}-1}]+I_{3}[u|u|^{p_{1}-1}%
]+I_{4}[u|u|^{p_{1}-1}]
\]
Defining the rescaled function $u_{\lambda}(x,t)=\lambda^{\frac{1}{p_{2}-1}%
}u(\lambda x,\lambda t)$, and using the scaling properties of the kernels
\[
P(\lambda x^{\prime},\lambda x_{n})=\lambda^{1-n}P(x^{\prime},x_{n})\text{ and
}G(\lambda x,\lambda y)=\lambda^{2-n}G(x,y),
\]
it is not difficult to show that
\[
u_{\lambda}=\lambda^{\frac{1}{p_{2}-1}}u(\lambda x,\lambda t)=I_{1}%
[\lambda^{\frac{1}{p_{2}-1}}\varphi(\lambda x^{\prime})]+I_{2}[(u_{\lambda
})_{0}|(u_{\lambda})_{0}|^{p_{2}-1}]+I_{3}[u_{\lambda}|u_{\lambda}|^{p_{1}%
-1}]+I_{4}[u_{\lambda}|u_{\lambda}|^{p_{1}-1}]
\]
Under the scaling hypothesis $\varphi(x^{\prime})=\lambda^{\frac{1}{p_{2}-1}%
}\varphi(\lambda x^{\prime})$, this can be simplified to%
\[
u_{\lambda}=I_{1}[\varphi]+I_{2}[(u_{\lambda})_{0}|(u_{\lambda})_{0}%
|^{p_{2}-1}]+I_{3}[u_{\lambda}|u_{\lambda}|^{p_{1}-1}]+I_{4}[u_{\lambda
}|u_{\lambda}|^{p_{1}-1}]
\]
Additionally, we observe that $\left\Vert u_{\lambda}\right\Vert
_{\mathcal{X}}=\left\Vert u\right\Vert _{\mathcal{X}}$. Thus both $u$ and
$u_{\lambda}$ are solutions of (\ref{int1}) in $D_{2\varepsilon}$. By
uniqueness, $u=u_{\lambda}$, as required.

\qed

\textbf{Part (ii) - Axial symmetry. } Recall that a rotation matrix preserves
the Euclidean norm, that is, $|M(x^{\prime})|=|x^{\prime}|$, and satisfies
$\det\left(  D_{x^{\prime}}M(x^{\prime})\right)  =1$. Let $u\in
D_{2\varepsilon}$ be the unique solution of the integral equation (\ref{int1}).

For $u_{M}(x,t)=u(M(x^{\prime}),x_{n},t)$, the invariance properties
\[
P(M(x^{\prime}-y^{\prime}),x_{n})=P(x^{\prime}-y^{\prime},x_{n})\text{ and
}G((M(x^{\prime}),x_{n}),(M(y^{\prime}),y_{n}))=G(x,y)
\]
yield
\[
u_{M}=I_{1}[\varphi\circ M]+I_{2}[(u_{M})_{0}|(u_{M})_{0}|^{p_{2}-1}%
]+I_{3}[u_{M}|u_{M}|^{p_{1}-1}]+I_{4}[u_{M}|u_{M}|^{p_{1}-1}].
\]
Now, using the assumption $\varphi\circ M=\varphi$, we deduce that
\[
u_{M}=I_{1}[\varphi]+I_{2}[(u_{M})_{0}|(u_{M})_{0}|^{p_{2}-1}]+I_{3}%
[u_{M}|u_{M}|^{p_{1}-1}]+I_{4}[u_{M}|u_{M}|^{p_{1}-1}].
\]
Using Morrey norms are invariant under rotations $M,$ it follows that
$\left\Vert u_{M}\right\Vert _{\mathcal{X}}=\left\Vert u\right\Vert
_{\mathcal{X}}$ and hence $u_{M}\in D_{2\varepsilon}.$ Consequently, since $u$
and $u_{M}$ are solutions of (\ref{int1}) in $D_{2\varepsilon}$, the
uniqueness property implies that $u=u_{M}$.

\qed

\textbf{Part (iii) - Positivity. }Assume that $\varphi\geq0$ on $\mathbb{R}%
^{n-1}$ and there exists a set $U\subset\mathbb{R}^{n-1}$ of positive measure
such that $\varphi>0$ on $U$. Let $u$ be the unique solution in
$D_{2\varepsilon}$ for (\ref{int1}) which was obtained as a fixed point of
$\Phi$ (see (\ref{map-1})). It follows that $u$ is the limit in $\mathcal{X}$
of the Picard sequence $\{u^{(k)}\}_{k\in\mathbb{N}}$, defined by%
\begin{equation}
u^{(k+1)}=\Phi\lbrack u^{(k)}]\text{ with }u^{(1)}=I_{1}[\varphi].
\label{aux-seq-1}%
\end{equation}

Recall the notation $u_{0}=(u)_{0}=u|_{x_{n}=0}.$Since the Poisson kernel $P$
is strictly positive and $\varphi>0$ on a set $U\subset\mathbb{R}^{n-1}$ of
positive measure, we obtain that
\[
u^{(1)}(x,t)=\int_{\mathbb{R}^{n-1}}P(x^{\prime}-y^{\prime},x_{n}%
+t)\varphi(y^{\prime})\,dy^{\prime}\geq\int_{U}P(x^{\prime}-y^{\prime}%
,x_{n}+t)\varphi(y^{\prime})\,dy^{\prime}>0,
\]
for a.e. $(x,t)\in\mathbb{R}_{+}^{n}\times(0,\infty)$. Also, $(u^{(1)}%
)_{0}(x^{\prime},t)>0$ a.e. in $\mathbb{R}^{n-1}\times(0,\infty).$ Note also
that the kernels of the integral operators $I_{j}$ in (\ref{int1}) are
nonnegative. It follows that $\Phi\lbrack v]\geq\Phi\lbrack w]$ a.e in
$\mathbb{R}_{+}^{n}\times(0,\infty)$ and $(\Phi\lbrack v])_{0}\geq(\Phi\lbrack
w])_{0}$ a.e. in $\mathbb{R}^{n-1}\times(0,\infty),$ when $v\geq w\geq0$ a.e
in $\mathbb{R}_{+}^{n}\times(0,\infty)$ and $v_{0}\geq w_{0}\geq0$ a.e. in
$\mathbb{R}^{n-1}\times(0,\infty).$

Now using the positivity properties of $u^{(1)}$ and that $u^{(2)}=\Phi\lbrack
u^{(1)}]$ and $\Phi\lbrack0]=u^{(1)},$ we obtain $u^{(2)}\geq$ $u^{(1)}>0$
a.e. in $\mathbb{R}_{+}^{n}\times(0,\infty)$ and $(u^{(2)})_{0}\geq$
$(u^{(1)})_{0}>0$ a.e. in $\mathbb{R}^{n-1}\times(0,\infty).$ Proceeding
inductively, for each $k\geq2,$ we conclude that
\[
u^{(k+1)}=\Phi\lbrack u^{(k)}]\geq\Phi\lbrack u^{(k-1)}]=u^{(k)}>0,\text{ a.e.
in }\mathbb{R}_{+}^{n}\times(0,\infty),
\]
and $(u^{(k+1)})_{0}\geq(u^{(k)})_{0}>0$ a.e. in $\mathbb{R}^{n-1}%
\times(0,\infty).$ As $u^{(k)}\rightarrow u$ in $\mathcal{X}$, we have (up to
a subsequence) the pointwise convergence almost everywhere in $\mathbb{R}%
_{+}^{n}\times(0,\infty)$ and in $\mathbb{R}^{n-1}\times(0,\infty).$ This
implies the corresponding nonnegativity of $u$ and $u_{0}.$ It follows that
$u=\Phi\lbrack u]\geq$ $\Phi\lbrack0]=u^{(1)}>0$ a.e. in $\mathbb{R}_{+}%
^{n}\times(0,\infty),$ and $u_{0}=(\Phi\lbrack u])_{0}\geq(\Phi\lbrack
0])_{0}=(u^{(1)})_{0}>0$ a.e. in $\mathbb{R}^{n-1}\times(0,\infty)$, as desired.

\qed

\subsection{Proof of Theorem \ref{teo-asy}}

\hspace{0.3cm}\label{Sec5.3}For simplicity, we define, for each $t>0,$ the
auxiliary norm
\begin{equation}
\left\Vert u(\cdot,t)\right\Vert _{X_{t}}=t^{\alpha}\left\Vert u(\cdot
,t)\right\Vert _{\mathcal{M}_{q_{1},\mu}(\mathbb{R}_{+}^{n})}+t^{\beta
}\left\Vert u(\cdot,0,t)\right\Vert _{\mathcal{M}_{q_{2},\mu}(\mathbb{R}%
^{n-1})}+\left\Vert u(\cdot,t)\right\Vert _{\mathcal{M}_{q_{0},\mu}%
(\mathbb{R}_{+}^{n})}. \label{aux-norm-t}%
\end{equation}
From assumption (\ref{cond-asymp1}), we have that
\begin{equation}
\lim_{t\rightarrow\infty}\left\Vert I_{1}[\varphi-\psi](\cdot,t)\right\Vert
_{X_{t}}=0. \label{cond-Ht-1}%
\end{equation}

Taking the difference between the integral equations satisfied by $u$ and $v$,
evaluating the norm (\ref{aux-norm-t}), computing the $\lim\sup_{t\rightarrow
\infty}$ in the resulting inequality, and then using (\ref{cond-Ht-1}), we
arrive at
\begin{align}
\lim\sup_{t\rightarrow\infty}\left\Vert (u-v)(\cdot,t)\right\Vert _{X_{t}}  &
\leq\lim\sup_{t\rightarrow\infty}\left\Vert I_{2}[u_{0}|u_{0}|^{p_{2}%
-1}](\cdot,t)-I_{2}[v_{0}|v_{0}|^{p_{2}-1}](\cdot,t)\right\Vert _{X_{t}%
}\nonumber\\
&  +\lim\sup_{t\rightarrow\infty}\left\Vert I_{3}[u|u|^{p_{1}-1}%
](\cdot,t)-I_{3}[v|v|^{p_{1}-1}](\cdot,t)\right\Vert _{X_{t}}\nonumber\\
&  +\lim\sup_{t\rightarrow\infty}\left\Vert I_{4}[u|u|^{p_{1}-1}%
](\cdot,t)-I_{4}[v|v|^{p_{1}-1}](\cdot,t)\right\Vert _{X_{t}}.
\label{aux-est-asy-1}%
\end{align}

We now estimate each term on the right-hand side of (\ref{aux-est-asy-1}).
First recall that the solutions $u$ and $v$ obtained in Theorem \ref{theo},
corresponding respectively to the data $\varphi$ and $\psi$, satisfy
\begin{equation}
\left\Vert u\right\Vert _{\mathcal{X}},\left\Vert v\right\Vert _{\mathcal{X}%
}\leq2\varepsilon. \label{aux-bound-sol}%
\end{equation}
Using estimates (\ref{c_lineares_2_1}), (\ref{c_lineares_2_2}) and
(\ref{c_lineares_2_3}) with $g(x^{\prime},t)$ as in (\ref{def-g-1}), and
making the change of the variable $s=rt,$ we obtain that
\begin{align}
&  \left\Vert I_{2}[u_{0}|u_{0}|^{p_{2}-1}](\cdot,t)-I_{2}[v_{0}|v_{0}%
|^{p_{2}-1}](\cdot,t)\right\Vert _{X_{t}}=\left\Vert I_{2}\left[  g\right]
(\cdot,t)\right\Vert _{X_{t}}\nonumber\\
&  \leq\int_{0}^{1}\left(  \bar{C}_{2,1}\frac{r^{-\beta p_{2}}}{(1-r)^{\alpha
-\beta p_{2}-1}}+\bar{C}_{2,2}\frac{r^{-\beta p_{2}}}{(1-r)^{1-\beta(p_{2}%
-1)}}+\bar{C}_{2,3}\frac{r^{-\beta p_{2}}}{(1-r)^{1-\beta p_{2}}}\right)
\nonumber\\
&  \qquad
\text{\ \ \ \ \ \ \ \ \ \ \ \ \ \ \ \ \ \ \ \ \ \ \ \ \ \ \ \ \ \ \ \ \ \ \ \ \ \ \ \ \ \ \ \ \ \ \ \ \ \ \ }%
\times\left(  (rt)^{\beta p_{2}}\Vert g(\cdot,rt)\Vert_{\mathcal{M}%
_{\frac{q_{2}}{p_{2}},\mu}(\mathbb{R}^{n-1})}\right)  dr.
\label{aux-stab-I2-1}%
\end{align}
Making $t\rightarrow\infty$ in the inequality (\ref{aux-stab-I2-1}) and using
(\ref{c_nonlinear_2_1}) and (\ref{aux-bound-sol}), we deduce that%
\begin{equation}
\lim\sup_{t\rightarrow\infty}\left\Vert I_{2}\left[  u_{0}|u_{0}|^{p_{2}%
-1}\right]  (\cdot,t)-I_{2}\left[  v_{0}|v_{0}|^{p_{2}-1}\right]
(\cdot,t)\right\Vert _{X_{t}}\leq2^{p_{2}}\varepsilon^{p_{2}-1}C_{2}%
p_{2}\left(  \lim\sup_{t\rightarrow\infty}t^{\beta}\left\Vert (u-v)(\cdot
,0,t)\right\Vert _{\mathcal{M}_{q_{2},\mu}(\mathbb{R}^{n-1})}\right)
\label{eq:asint_1}%
\end{equation}

Next we treat the parcel with $I_{3}$ in (\ref{aux-est-asy-1}). For that, we
consider $f(x,t)$ as in (\ref{def-f-1}) and proceed similarly to the proofs of
estimates (\ref{c_lineares_3_1}), (\ref{eq:est_lineales_3.2}) and
(\ref{c_lineares_3_3}) to obtain
\begin{align}
&  \lim\sup_{t\rightarrow\infty}\left\Vert I_{3}[u|u|^{p_{1}-1}]-I_{3}%
[v|v|^{p_{1}-1}](\cdot,t)\right\Vert _{X_{t}}=\lim\sup_{t\rightarrow\infty
}\left\Vert I_{3}[f](\cdot,t)\right\Vert _{X_{t}}\nonumber\\
&  \leq\lim\sup_{t\rightarrow\infty}\int_{0}^{1}\left(  C\frac{r^{-\alpha
p_{1}}}{(1-r)^{1-\alpha(p_{1}-1)}}+C\frac{r^{-\alpha p_{1}}}{(1-r)^{1-\alpha
p_{1}+\beta}}+C\frac{r^{-\alpha p_{1}}}{(1-r)^{1-\alpha p_{1}}}\right)
\nonumber\\
&  \qquad
\text{\ \ \ \ \ \ \ \ \ \ \ \ \ \ \ \ \ \ \ \ \ \ \ \ \ \ \ \ \ \ \ \ \ \ \ \ \ \ \ \ \ \ \ \ \ \ \ \ \ \ \ \ \ \ }%
\times\left(  (rt)^{\alpha p_{1}}\Vert f(\cdot,rt)\Vert_{\mathcal{M}%
_{\frac{q_{1}}{p_{1}},\mu}(\mathbb{R}_{+}^{n})}\right)  dr.\nonumber\\
&  \leq C_{3}p_{1}\left(  \lim\sup_{t\rightarrow\infty}t^{\alpha}\left\Vert
(u-v)(\cdot,t)\right\Vert _{\mathcal{M}_{q_{1},\mu}(\mathbb{R}_{+}^{n}%
)}\right)  \left[  \left(  2\varepsilon\right)  ^{p_{1}-1}+\left(
2\varepsilon\right)  ^{p_{1}-1}\right] \nonumber\\
&  =2^{p_{1}}\varepsilon^{p_{1}-1}C_{3}p_{1}\left(  \lim\sup_{t\rightarrow
\infty}t^{\alpha}\left\Vert (u-v)(\cdot,t)\right\Vert _{\mathcal{M}_{q_{1}%
,\mu}(\mathbb{R}_{+}^{n})}\right)  , \label{eq:asint_2}%
\end{align}
where above we have used (\ref{c_nonlinear_3_1}) and (\ref{aux-bound-sol}).

For the term $I_{4},$ we consider again $f(x,t)$ as in (\ref{def-f-1}), and
proceed analogously to the derivations of (\ref{eq:est_lineales_4.1}),
(\ref{eq:est_lineales_4.2}) and (\ref{eq:est_lineales_4.3}). Then, by applying
(\ref{c_nonlinear_4_1}), (\ref{c_nonlinear_4_2}) and (\ref{aux-bound-sol}), we
arrive at
\begin{align}
&  \lim\sup_{t\rightarrow\infty}\left\Vert I_{4}[u|u|^{p_{1}-1}](\cdot
,t)-I_{4}[v|v|^{p_{1}-1}](\cdot,t)\right\Vert _{X_{t}}=\lim\sup_{t\rightarrow
\infty}\left\Vert I_{4}[f](\cdot,t)\right\Vert _{X_{t}}\nonumber\\
&  \leq C_{4,1}p_{1}\left(  \lim\sup_{t\rightarrow\infty}t^{\alpha}\left\Vert
(u-v)(\cdot,t)\right\Vert _{\mathcal{M}_{q_{1},\mu}(\mathbb{R}_{+}^{n}%
)}\right)  \left[  \left(  2\varepsilon\right)  ^{p_{1}-1}+\left(
2\varepsilon\right)  ^{p_{1}-1}\right] \nonumber\\
&  +C_{4,2}p_{1}\left(  \lim\sup_{t\rightarrow\infty}\left\Vert (u-v)(\cdot
,t)\right\Vert _{\mathcal{M}_{q_{0},\mu}(\mathbb{R}_{+}^{n})}\right)  \left[
\left(  2\varepsilon\right)  ^{p_{1}-1}+\left(  2\varepsilon\right)
^{p_{1}-1}\right]  .\nonumber\\
&  =2^{p_{1}}\varepsilon^{p_{1}-1}C_{4,1}p_{1}\left(  \lim\sup_{t\rightarrow
\infty}t^{\alpha}\left\Vert (u-v)(\cdot,t)\right\Vert _{\mathcal{M}_{q_{1}%
,\mu}(\mathbb{R}_{+}^{n})}\right) \nonumber\\
&  +2^{p_{1}}\varepsilon^{p_{1}-1}C_{4,2}p_{1}\left(  \lim\sup_{t\rightarrow
\infty}\left\Vert (u-v)(\cdot,t)\right\Vert _{\mathcal{M}_{q_{0},\mu
}(\mathbb{R}_{+}^{n})}\right)  . \label{eq:asint_3}%
\end{align}
Finally, defining $\Gamma=\lim\sup_{t\rightarrow\infty}\left\Vert
(u-v)(\cdot,t)\right\Vert _{X_{t}}$ and putting together estimates
(\ref{aux-est-asy-1}), (\ref{eq:asint_1}), (\ref{eq:asint_2}) and
(\ref{eq:asint_3}) leads us to%
\begin{align}
0  &  \leq\Gamma=\lim\sup_{t\rightarrow\infty}\left\Vert (u-v)(\cdot
,t)\right\Vert _{X_{t}}\nonumber\\
&  \leq2^{p_{2}}\varepsilon^{p_{2}-1}C_{2}p_{2}\Gamma+2^{p_{1}}\varepsilon
^{p_{1}-1}C_{3}p_{1}\Gamma+2^{p_{1}}\varepsilon^{p_{1}-1}C_{4,1}p_{1}%
\Gamma+2^{p_{1}}\varepsilon^{p_{1}-1}C_{4,2}p_{1}\Gamma\nonumber\\
&  =[p_{2}C_{2}2^{p_{1}}\varepsilon^{p_{1}-1}+p_{1}(C_{3}+C_{4})2^{p_{1}%
}\varepsilon^{p_{1}-1}]\Gamma\nonumber\\
&  =\eta\Gamma\label{aux-stab-10}%
\end{align}
where $C_{4}=C_{4,1}+C_{4,2}$ and
\[
0<\eta=p_{2}C_{2}2^{p_{2}}\varepsilon^{p_{2}-1}+p_{1}(C_{3}+C_{4})2^{p_{1}%
}\varepsilon^{p_{1}-1}<1\text{ (by (\ref{q_menorq1})).}%
\]
This implies that $\Gamma=0$ and we are done.

\qed

\

\bigskip

\noindent\textbf{Conflict of Interest Statement:} The authors declare that
there are no conflicts of interest to disclose.

\

\noindent\textbf{Data availability statement:} This manuscript has no
associated data.


\bigskip

\begin{thebibliography}{99}                                                                                               %


\bibitem {adams2015morrey}D.~R. Adams, Morrey Spaces, Lecture Notes in Applied
and Numerical Harmonic Analysis, \textit{Birkh\"{a}user/Springer, Cham}, 2015.

\bibitem {almeida2018approximation}A.~Almeida and S.~Samko, Approximation in
Morrey spaces. \textit{Journal of Functional Analysis} \textbf{272} (6)
(2017), 2392--2411.

\bibitem {Fila1997}H.~Amann and M.~Fila, A Fujita-type theorem for the Laplace
equation with a dynamical boundary condition. \textit{Acta Math. Univ.
Comenian. (N.S.)} \textbf{66} (2) (1997), 321--328.

\bibitem {Arrieta-DIE2001}J. M. Arrieta, P. Quittner, and A.
Rodr\'{\i}guez-Bernal, Parabolic problems with nonlinear dynamical boundary
conditions and singular initial data. \textit{Differential Integral Equations}
\textbf{14} (12) (2001), 1487--1510.

\bibitem {Binz-2024}T. Binz and A. F. M. ter Elst, Dynamic boundary conditions
for divergence form operators with H\"{o}lder coefficients. \textit{Ann. Sc.
Norm. Super. Pisa Cl. Sci.} \textbf{25} (4) (2024), 2173--2199.

\bibitem {Crank1975}J. Crank, The mathematics of diffusion. Second edition.
\textit{Clarendon Press, Oxford,} 1975.

\bibitem {Ferreira}M.~F. de Almeida and L.~C.~F. Ferreira, On the
Navier-Stokes equations in the half-space with initial and boundary rough data
in Morrey spaces. \textit{J. Differential Equations} \textbf{254} (3) (2013), 1548--1570.

\bibitem {diaz1984aplicacion}J.~I.~D\'{\i}az and R.~Jim\'{e}nez,
Aplicaci\'{o}n de la teor\'{\i}a no lineal de semigrupos a un operador
pseudodiferencial. \textit{Actas VII CEDYA}, University of Granada, Spain,
137--142, 1984.

\bibitem {escher1992nonlinear}J.~Escher, Nonlinear elliptic systems with
dynamic boundary conditions. \textit{Math. Z.} \textbf{210} (1992), 413--439.

\bibitem {FerreiraSantana2024}L.~C.~F. Ferreira and M.~G. Santana, On the
Helmholtz decomposition in Morrey and block spaces. \textit{Math. Ann.}
\textbf{392} (3) (2025), 4315--4360.

\bibitem {Ferreira2}L. C. F. Ferreira, On a bilinear estimate in weak-Morrey
spaces and uniqueness for Navier-Stokes equations. \textit{J. Math. Pures
Appl.} 105 (2) (2016), 228--247.

\bibitem {fila2013large}M.~Fila, K.~Ishige, and T.~Kawakami, Large-time
behavior of solutions of a semilinear elliptic equation with a dynamical
boundary condition. \textit{Adv. Differential Equations} \textbf{18} (1--2)
(2013), 69--100.

\bibitem {fila2015existence}M.~Fila, K.~Ishige, and T.~Kawakami, Existence of
positive solutions of a semilinear elliptic equation with a dynamical boundary
condition. \textit{Calc. Var. Partial Differential Equations} \textbf{54} (2)
(2015), 2059--2078.

\bibitem {fila2016minimal}M.~Fila, K.~Ishige, and T.~Kawakami, Minimal
solutions of a semilinear elliptic equation with a dynamical boundary
condition. \textit{J. Math. Pures Appl.} \textbf{105} (6) (2016), 788--809.

\bibitem {fila2017exterior}M.~Fila, K.~Ishige, and T.~Kawakami, An exterior
nonlinear elliptic problem with a dynamical boundary condition. \textit{Rev.
Mat. Complut.} \textbf{30} (2) (2017), 281--312.

\bibitem {fila1997global}M.~Fila and P.~Quittner, Global solutions of the
Laplace equation with a nonlinear dynamical boundary condition. \textit{Math.
Methods Appl. Sci.} \textbf{20} (15) (1997), 1325--1333.

\bibitem {folland1999real}G.~B. Folland, Real Analysis. Modern techniques and
their applications. Second edition. Pure Appl. Math. (N. Y.). Wiley-Intersci.
Publ. \textit{John Wiley \& Sons, Inc., New York}, 1999.

\bibitem {giga2018dynamic}Y.~Giga and N.~Hamamuki, On a dynamic boundary
condition for singular degenerate parabolic equations in a half space,
\textit{NoDEA Nonlinear Differential Equations Appl.} \textbf{25} (6) (2018),
Paper No.~51, 39 pp.

\bibitem {Giga-CPDE-1989}Y. Giga and T. Miyakawa, Navier-Stokes flow in
$\mathbb{R}^{3}$ with measures as initial vorticity and Morrey spaces.
\textit{Comm. Partial Differential Equations} \textbf{14} (5) (1989), 577--618.

\bibitem {Giga-DIE-2021}Y. Giga, F. Onoue, and K. Takasao, A varifold
formulation of mean curvature flow with Dirichlet or dynamic boundary
conditions. \textit{Differential Integral Equations} \textbf{34} (1-2) (2021), 21--126.

\bibitem {hintermann1989evolution}T.~Hintermann, Evolution equations with
dynamic boundary conditions. \textit{Proc. Roy. Soc. Edinburgh Sect. A}
\textbf{113} (1--2) (1989), 43--60.

\bibitem {Kato1992}T.~Kato, Strong solutions of the Navier-Stokes equation in
Morrey spaces. \textit{Bol. Soc. Brasil. Mat. (N.S.)} \textbf{22} (2) (1992), 127--155.

\bibitem {kirane1992blow}M.~Kirane, Blow-up for some equations with semilinear
dynamical boundary conditions of parabolic and hyperbolic type,
\textit{Hokkaido Math. J.} \textbf{21} (2) (1992), 221--229.

\bibitem {kirane2004nonexistence}M.~Kirane, E.~Nabana, and S.~I. Pokhozhaev,
Nonexistence of global solutions to an elliptic equation with a dynamical
boundary condition, \textit{Bol. Soc. Paran. Mat.} \textbf{22} (2) (2004), 9--16.

\bibitem {Latorre-Segura2018}M. Latorre and S. Segura de Le\'{o}n, Elliptic
1-Laplacian equations with dynamical boundary conditions. \textit{J. Math.
Anal. Appl.} \textbf{464} (2) (2018), 1051--1081.

\bibitem {lions1969quelques}J.-L. Lions, Quelques M\'{e}thodes de
R\'{e}solution des Probl\`{e}mes aux Limites Non-Lin\'{e}aires. \textit{Dunod,
Paris; Gauthier-Villars, Paris}, 1969.

\bibitem {Liu}L.~Liu and J.~Xiao, A trace law for the Hardy-Morrey-Sobolev
space. \textit{J. Funct. Anal.} \textbf{274} (1) (2018), 80--120.

\bibitem {Liu_xiao}L.~Liu and J.~Xiao, Restricting Riesz-Morrey-Hardy
potentials. \textit{J. Differential Equations} \textbf{262} (11) (2017), 5468--5496.

\bibitem {Stein}E.~M. Stein, Singular integrals and differentiability
properties of functions. Princeton Mathematical Series, No. 30.
\textit{Princeton University Press, Princeton, NJ}, 1970.
\end{thebibliography}
\end{document}